\RequirePackage{fix-cm}
\documentclass[smallextended, envcountsect]{svjour3}       
\smartqed  
\usepackage{graphicx,booktabs}
\usepackage[misc]{ifsym}

\usepackage{fixmath}
\usepackage{psfrag, subfigure, bm}
\usepackage[usenames,dvipsnames]{pstricks}
\usepackage{amsmath,amsfonts,amssymb}
\usepackage{mathabx}
\usepackage[numbers, sort&compress]{natbib}

\usepackage[bottom = 3cm]{geometry}
\usepackage{latexsym}
 \renewcommand{\a}{\mathbold{a}}
 \newcommand{\transp}{\mathsf{T}}
 
 \newcommand{\x}{\mathbold{x}}
  \newcommand{\e}{\mathbold{e}}
 
 \newcommand{\X}{\mathbold{X}}

 \newcommand{\g}{\mathbold{g}}
 
 \newcommand{\y}{\mathbold{y}}
 
 \renewcommand{\c}{\mathbold{c}}
 \newcommand{\bbeta}{\mathbold{\beta}}
 \newcommand{\z}{\mathbold{z}}
 \newcommand{\A}{\mathbold{A}}
 \newcommand{\B}{\mathbold{B}}
 \renewcommand{\b}{\mathbold{b}}

 \DeclareMathOperator*{\minimize}{\mathrm{minimize}}
 \DeclareMathOperator*{\maximize}{\mathrm{maximize}}

 \newtheorem{assumption}{Assumption}[section]
 
\newcommand\bnu{\mathbold{\nu}}
\newcommand\cf{f^{\star}}
\newcommand\cg{g^{\star}}

\newcommand\p{\mathbold{p}}
\newcommand\q{\mathbold{q}}

\newcommand{\fix}{\mathrm{fix}}
\newcommand{\zer}{\mathrm{zer}}
\newcommand{\im}{\mathrm{im}}
\newcommand{\prox}{\mathrm{prox}}

\newcommand{\opT}{\mathsf{T}}
\newcommand{\opG}{\mathsf{G}}
\newcommand{\opI}{\mathsf{I}}
\newcommand{\opR}{\mathsf{R}}
\newcommand{\opC}{\mathsf{C}}

\newcommand{\MK}{Mann-Krasnosel'skii}
\newcommand{\Mk}{Mann-Krasnosel'skii\ }

\begin{document}

\title{Time-Varying Convex Optimization via Time-Varying Averaged Operators
}

\titlerunning{Time-Varying Convex Optimization}        

\author{Andrea Simonetto         
}


\institute{A. Simonetto (\Letter) \at
              IBM Research Ireland, Dublin, Ireland. \\
              Tel.: +353-87-3583843\\
              \email{andrea.simonetto@ibm.com}           
}

\date{Received: date / Accepted: date}

\maketitle

\begin{abstract}
Devising efficient algorithms that track the optimizers of continuously varying convex optimization problems is key in many applications. A possible strategy is to sample the time-varying problem at constant rate and solve the resulting time-invariant problem. This can be too computationally burdensome in many scenarios. An alternative strategy is to set up an iterative algorithm that generates a sequence of approximate optimizers, which are refined every time a new sampled time-invariant problem is available by one iteration of the algorithm. This type of algorithms are called running. A major limitation of current running algorithms is their key assumption of strong convexity and strong smoothness of the time-varying convex function. In addition, constraints are only handled in simple cases. This limits the current capability for running algorithms to tackle relevant problems, such as $\ell_1$-regularized optimization programs. In this paper, these assumptions are lifted by leveraging averaged operator theory and a fairly comprehensive framework for time-varying convex optimization is presented. In doing so, new results characterizing the convergence of running versions of a number of widely used algorithms are derived.


\vskip0.2cm

\end{abstract}
\newpage

\section{Introduction}

The goal of this paper is to present a unifying view on time-varying convex optimization based on the theory of averaged operators.  Time-varying convex optimization has appeared as a natural extension of convex optimization where the cost function, the constraints, or both, depend on a time parameter and change continuously in time. This setting captures relevant control problems~\cite{Jerez2014,Hours2014,Gutjahr2016}, when, for instance, one is interested in generating a control action depending on a (parametric) varying optimization problem, as well as signal processing problems~\cite{Jakubiec2013}, where one seeks to estimate a dynamical process based on time-varying observations, or in time-varying compressive sensing~\cite{Asif2014,Yang2015,Vaswani2015,Balavoine2015,Simonetto2015a,Sopasakis2016} and inferential problems on dynamic networks~\cite{Baingana2015}. Additional application domains include  robotics~\cite{ardeshiri2010convex,verscheure2009time,Koppel2015a}, smart grids~\cite{Zhao2014,DallAnese2016}, economics~\cite{Dontchev2013}, and real-time magnetic resonance imaging (MRI)~\cite{Uecker2012}. In the (big) data analytics community, time-varying optimization is appearing in stream computing. 

It is therefore of the utmost importance to present a theory that can encompass the most general optimization problems, and derive algorithms that find and track the optimizer sets of such continuously varying problems. The task of designing such algorithms is usually split into two phases: in the first phase one samples the time-varying optimization problem at discrete sampling instances, so to obtain a time-invariant problem. The second phase is the construction of near optimal decision variables for the time-invariant problems. When the sampling period is small enough, then one can reconstruct the solution trajectory (i.e., the decision variables as a function of time),  with arbitrary accuracy. 

It is rather clear to see that, when each instance of the problem is of large-scale, or when it involves the communication over a network of computing nodes (in a distributed setting), finding accurate near optimal decision variables for each instance is a daunting task. In practice, one would instead attempt at designing algorithms that run at the same time of the changes in the optimization problems. Think of the gradient method for unconstrained optimization. If the cost function changes in time, one would like to sample the cost function and perform only a few (perhaps only one) gradient step(s) per sampling time. This in contrast with the computationally harder task of running the gradient method at optimality for each instance of the problem. We call the algorithms that perform a limited number of iteration per sampling time \emph{running} methods.   

At the present stage, \emph{running} methods have been derived for special classes of optimization problems, namely strong convex and strong smooth cost functions with no or simple constraint sets~\cite{Popkov2005,Tu2011,Bajovic2011, Dontchev2013,Zavlanos2013,Jakubiec2013,Ling2013,Simonetto2014c,Simonetto2014d,Ye2015,Xi2016a,Sun2017}. A very interesting recent paper~\cite{Maros2017} has presented a running alternating direction method of multipliers (ADMM) algorithm that has been proven to converge even if the decomposed problems are not necessarily strongly convex. The proof technique relies on a compactness assumption of the feasible set. 

In some cases, requiring higher order smoothness conditions, prediction-correction schemes have been implemented~\cite{Paper1,Paper2,Paper3}, where not only the algorithm react to the changes in the problem, but actively predict how the optimal decision variable evolve.  Some works, under these smooth and strong convexity settings, have proposed continuous-time algorithms~\cite{Rahili2015,Fazlyab2015,Fazlyab2016}. 

In this paper, we use the theory of averaged operators~\cite{Bauschke2011} to derive running algorithms for a larger class of time-varying optimization problems and by doing so we generalize a number of results that have appeared in recent years. In particular, 
\begin{enumerate}
\item[\emph{i)}] We propose a running version of the \Mk (fixed-point) iteration and prove its convergence under reasonable assumptions (Theorems~\ref{theo:1} and~\ref{theo:3}). The time-invariant version of this iteration is the building block of a very large class of time-invariant optimization algorithms; similarly, the running version is key for time-varying ones;
\item[\emph{ii)}] We present the consequences of the running \Mk iteration on time-varying optimization. We derive a number of algorithms, namely running projected gradient, proximal-point, forward-backward splitting, and dual ascent and prove their convergence (Corollaries~\ref{co:gradient} till~\ref{co:da} and Proposition~\ref{co:fb_obj}, Corollaries~\ref{co:gr-obj}-\ref{co:pp-obj}). These results extend the work in~\cite{Popkov2005, Simonetto2014c} to a wider class of optimization problems;
\item[\emph{iii)}] We show how to enforce a properly defined bounded assumption for an even larger class of time-varying optimization problems, and this allows us to derive the running versions of both dual decomposition and ADMM and prove their convergence (Corollaries~\ref{co:daeq}-\ref{co:admm}). These results are important generalizations of earlier works \cite{Jakubiec2013, Ling2013}.  
\end{enumerate} 

The remainder of the paper is organized as follows. Section~\ref{sec:prelim} presents some necessary preliminaries on averaged operators and on the \Mk iteration in the time-invariant setting. In Section~\ref{sec:main}, we state the main assumptions, propose the running \Mk iteration, and prove its convergence. Section~\ref{sec:alternatives} reports an alternative problem assumption to the ones presented in Section~\ref{sec:main} and offers a different angle to tackle the convergence proof for the running \Mk iteration. Sections~\ref{sec:vanishing} and \ref{sec:literature} are somewhat additional, but still relevant, and they study the case of a time-varying setting that eventually reaches steady-state, and provide links to existing works, respectively. Sections~\ref{sec:opt}, \ref{sec:opt.b}, and \ref{sec:obj_conv} focus on the consequences of the running \Mk iteration on time-varying optimization, which is the main aim of this paper. A numerical example is offered in Section~\ref{sec:num} and we conclude in Section~\ref{sec:concl}.

\textbf{Notation. } Vectors and matrices are indicated in boldface, e.g., $\x\in\mathbb{R}^n$, $\A \in \mathbb{R}^{n \times m}$, sets with calligraphic letters as $\mathcal{X}$.  We use $\|\cdot\|$ to denote the Euclidean norm in the vector space, and the respective induced norms for matrices and tensors. The norm of a set $\mathcal{X}$ is the norm of its largest element w.r.t the selected vector/matrix norm.   

We will deal with time-varying functions $f(\x; t): \mathbb{R}^n \times \mathbb{R}_{+} \to \mathbb{R}$, whose properties are said to be uniform if they are true for all times $t$. For example, a function $f(\x; t)$ is said to be uniformly convex, iff it is convex in the variable $\x$ for all $t$. 

A function $f(\x)$ is strongly convex with constant $m$, iff $f(\x)- m/2\|\x\|^2$ is convex. A function $f(\x)$ is strongly smooth (or equivalently is differentiable and has Lipschitz continuous gradient) with constant M, iff $f(\x) - M/2\|\x\|^2$ is concave (Other equivalent definitions can be used, see~\cite{Ryu2015}). We indicate the subdifferential operator of a convex function $f$ as $\partial f$, which is defined as
$$
\partial f (\x) = \{ \g | \g^\transp (\y-\x) \leq f(\y)-f(\x), \forall \y \in \textrm{dom} f\};
$$
when the function is differentiable, then $\partial f = \nabla f$, that is the subdifferential operator is the gradient operator. Subdifferential operators are in general set-valued operators, while gradient operators are single-valued. Functions that are closed, convex, and proper are indicated as CCP. The Fenchel's conjugate of a function $f:\mathbb{R}^n \to \mathbb{R}$ is indicated with $\cf$ and has the usual definition $\cf(\y) = \sup_{\x\in\mathbb{R}^n}\{\y^\transp\x - f(\x)\}$. 

Operators are indicated with capital sans serif letters like $\opT$, or $\opI$ for the identity operator.


\section{Preliminaries}\label{sec:prelim}

Some necessary preliminaries are reviewed in this section; the interested readers can find more details in standard references such as \cite{Rockafellar1970, Eckstein1989, Rockafellar1998, Bauschke2011, AragonArtacho2014a, Ryu2015}. 

A set-valued operator $\opT: \mathbb{R}^n \rightrightarrows \mathbb{R}^n$ is said to be monotone, if it satisfies
\begin{equation}
(\opT(\x) - \opT(\y))^\transp(\x - \y) \geq 0,\quad \forall \x, \y \in \mathbb{R}^n.
\end{equation} 
A monotone operator is maximal if there is no monotone operator that properly contains it.  An example of maximal monotone operator is the subdifferential $\partial f$ of a closed convex proper function $f$. An operator $\opT: \mathbb{R}^n \rightrightarrows \mathbb{R}^n$ is said to be a contraction, if
\begin{equation}\label{eq.contr}
\|\opT(\x) - \opT(\y)\| \leq L \|\x - \y\|,\quad \forall \x, \y \in \mathbb{R}^n
\end{equation}
for $L \in (0,1)$. If $L = 1$, $\opT$ is said to be nonexpansive. Both cases imply that $\opT$ is a function. An operator $\opT: \mathbb{R}^n \rightrightarrows \mathbb{R}^n$ is said to be $\alpha$-averaged (or simply averaged) when it is the convex combination of a nonexpansive operator $\opG: \mathbb{R}^n \rightrightarrows \mathbb{R}^n$ and the identity operator $\opI$, i.e., 
\begin{equation}\label{eq.alphaav}
\opT = (1-\alpha) \opI + \alpha \opG, 
\end{equation}  
for $\alpha \in (0,1)$. Averaged operators are nonexpansive by construction. 

\begin{proposition}\emph{(Composition of $\alpha$-averaged operators)}\label{composition}\cite[Proposition 2.4]{Combettes2015}
Let $\opT_1: \mathbb{R}^n \rightrightarrows \mathbb{R}^n$ and $\opT_2: \mathbb{R}^n \rightrightarrows \mathbb{R}^n$ be two $\alpha$-averaged operators with constants $\alpha_1$ and $\alpha_2$, respectively. Define 
\begin{equation}
\opT = \opT_1 \opT_2, \quad \alpha = \frac{\alpha_1 + \alpha_2 - 2\alpha_1\alpha_2}{1-\alpha_1 \alpha_2}.
\end{equation}
Then the operator $\opT$ is $\alpha$-averaged with constant $\alpha$.
\end{proposition}


A fixed point of the operator $\opT$ is a point $\x \in \mathbb{R}^n$ for which $\x = \opT(\x)$. If $\opT$ is $\alpha$-averaged in the sense of~\eqref{eq.alphaav}, then $\opT$ and $\opG$ have the same fixed points, i.e., $\fix \opT = \fix \opG$. 

\begin{proposition}\emph{(\Mk iteration, \cite{Bauschke2011, Cominetti2014, Ryu2015})}\label{prop.Mk}
Consider the operator $\opT: \mathbb{R}^n \rightrightarrows \mathbb{R}^n$. Let $\opT$ be $\alpha$-averaged in the sense of~\eqref{eq.alphaav}. Consider the sequence $\{\x_k\}_{k \in \mathbb{N}_{>0}}$ generated by the \Mk (or fixed point) iteration
\begin{equation}\label{fixedpoint}
\x_{k+1} = \opT(\x_{k}) = \x_k + \alpha(\opG(\x_k) - \x_k).
\end{equation}
Then the sequence $\{\x_k\}_{k \in \mathbb{N}_{>0}}$ converges weakly to a fixed point of $\opT$, i.e. $\x_k \rightharpoonup \x^*$, with $\x^* \in \fix \opT$, and we have the following bounds on the fixed-point residual $\|\opG(\x_k) - \x_k\|$,
\begin{equation}
\frac{1}{T}\sum_{k=1}^T\|\opG(\x_k) - \x_k \|^2 \leq \frac{\|\x_1 - \x^*\|^2}{\alpha (1-\alpha)\, T}, \quad \|\opG(\x_T) - \x_T \| \leq \frac{\|\x_1 - \x^*\|}{\sqrt{\alpha (1-\alpha)\, T}},
\end{equation}
where $T$ is the number of iterations. 

In addition, if $\opT$ is a contraction with constant $L$, then the ``decision variables'' sequence $\{\x_k\}_{k \in \mathbb{N}_{>0}}$ converges strongly to a fixed point of $\opT$ as
\begin{equation}\label{linear}
\|\x_k - \x^*\| \leq L^{k-1} \|\x_1 - \x^*\|, \quad \x^* \in \fix \opT.
\end{equation}
\end{proposition}

Proposition~\ref{prop.Mk} is key in generating algorithms to find fixed points of operators $\opT$. If $\opT$ is $\alpha$-averaged, then by using~\eqref{fixedpoint} one can compute the fixed points of $\opT$ in the limit, and the convergence rate based on the squared of the residual is bounded as $O(1/T)$. We note that when the norm $\|\opG(\x_k) - \x_k \|\to 0$, then $\x_k$ is a fixed point of $\opG$ and thus of $\opT$. This type of error norm is useful in practice, since one can easily monitor it on-line and use it as a stopping criterion. The result on the convergence of $\|\opG(\x_k) - \x_k \|$ as $O(1/\sqrt{T})$ is due to~\cite{Vaisman2005} (see also~\cite{Cominetti2014}).

If in addition $\opT$ is a contraction, then one obtain linear convergence in the decision variables $\x_k$. This second convergence result is stronger, since it involves directly the decision variables. From~\eqref{linear}, one can also derive a stronger result on the residual as\footnote{From $\|\opT(\x_k) - \x_k \| = \|\opT(\x_k) - \opT(\x^*) + \x^* - \x_k \| \leq \|\opT(\x_k) - \opT(\x^*)\|  + \|\x_k - \x^*\| \leq 2 \|\x_k - \x^*\|$, and then applying~\eqref{linear} and squaring.}, 
\begin{equation}
\|\opT(\x_k) - \x_k \|^2 \leq 4\,L^{2(k-1)} \|\x_1 - \x^*\|^2, \quad \x^* \in \fix \opT.
\end{equation}

Finding fixed points is a cornerstone in convex optimization. The prototype problem, 
\begin{equation}\label{optprob}
\minimize_{\x \in \mathbb{R}^n} \, f(\x)
\end{equation} 
where $f: \mathbb{R}^n \to \mathbb{R}$ is a closed convex proper function, can be interpreted as finding the zeros of the subdifferential operator $\partial f$, which is equivalent of finding the fixed points of the operator $\opI - \lambda \partial f$, for all nonzero scalar $\lambda$, i.e.,
\begin{equation}
\x^* \in \zer\, \partial f \quad \iff \quad \x^* \in \fix (\opI - \lambda \partial f)\,. 
\end{equation} 
When $f$ is strongly smooth with parameter $M$, and thus $\partial f = \nabla f$, and $\lambda \in (0, 2/M)$, then the operator $\opT_{\mathrm{G}} = \opI - \lambda \nabla f$ is $\alpha$-averaged with $\alpha = \lambda M/2$~\cite{Ryu2015} and therefore the fixed point iteration
\begin{equation}\label{eq.gradient}
\x_{k+1} = (\opI - \lambda \nabla f)(\x_k) = \x_k - \lambda \nabla f(\x_k)
\end{equation}
generates a sequence $\{\x_k\}_{k\in\mathbb{N}_{>0}}$ that converges as dictated by Proposition~\ref{prop.Mk}. To see the $\alpha$-averageness of $\opT_{\mathrm{G}}$ is sufficient to notice that
\begin{equation}
\opI - \lambda \nabla f = (1-\alpha)\opI + \alpha (\opI - 2/M \nabla f), \quad \mathrm{with}\,\, \alpha = \lambda M/2.
\end{equation}
Therefore, as for Proposition~\ref{prop.Mk} one has that
\begin{equation}
\|\nabla f(\x_T)\| \leq O(1/\sqrt{T}).
\end{equation}
In addition, if $f$ is also strongly convex with parameter $m$, then $\opI - \lambda \nabla f$ is a contraction, and linear convergence, that is~\eqref{linear}, can be established as~\cite{Ryu2015} 
\begin{equation}\label{grlinear}
\|\x_k - \x^*\| \leq L^{k-1} \|\x_1 - \x^*\|, \quad L = \max\{|1-\lambda m|,|1 - \lambda M|\},
\end{equation}
where now $\x^*$ is the unique optimizer of~\eqref{optprob}. Iteration~\eqref{eq.gradient} is generally known as the gradient method. 

Many other convex optimization algorithms can be seen as fixed point iterations of a properly defined $\alpha$-averaged operators. To mention only a few, projected gradient method~\cite{Goldstein1964, Levitin1966}, proximal point method~\cite{Rockafellar1976}, iterative shrinkage thresholding algorithm (ISTA)~\cite{Beck2009}, dual ascent~\cite{Tseng1990,Nedic2009a}, forward-backward splitting~\cite{Combettes2005, Duchi2009}, and the celebrated alternating direction method of multipliers (ADMM)~\cite{Bertsekas1997,Schizas2008,Boyd2011} fall in this class. 

\section{Problem Formulation and Time-Varying Algorithm}\label{sec:main}

The main aim of this paper is to develop a more general theory for time-varying convex optimization, that is devising efficient algorithms capable of finding and tracking the solution set of continuously varying convex programs. 

In order to achieve this goal, operator theory is leveraged. In particular, as discussed, there is a tight connection between a large class of algorithms used in time-invariant optimization and finding the fixed points of careful designed $\alpha$-averaged operators. In this respect, in this section, the focus is on designing algorithms to find the fixed points of a continuously varying $\alpha$-averaged operator. The connections with optimization will be clear in Sections~\ref{sec:opt}-\ref{sec:opt.b}. 

The aim is therefore determining for each time $t\geq 0$, the set (or a point in the set)
\begin{equation}
\fix \opT(\x; t),
\end{equation}
where $\opT: \mathbb{R}^n \times \mathbb{R}_{+} \rightrightarrows \mathbb{R}^n$ is an operator uniformly in time. The approach is to sample the operator at discrete sampling times $t_k, k \in \mathbb{N}_{> 0}$, and to determine the time-invariant sets 
\begin{equation}
\fix \opT(\x; t_k) =: \opT_k(\x),
\end{equation}
for each sampling time $t_k$. 

The algorithms that are sought are of the form:
\begin{enumerate}
\item[1.] Set $\x_1$ arbitrarily,
\item[2.] \textbf{for} $k>1$ \textbf{do}:

Sample the time-varying operator $\opT_k(\cdot) = \opT(\cdot; t_k)$;

Compute the next approximate fixed point
\begin{equation}\label{tvfixedpoint}
\x_{k+1} = \opT_k(\x_k).
\end{equation}
\end{enumerate}

In accordance with widespread nomenclature, Iteration~\eqref{tvfixedpoint} is called the \emph{running} \Mk algorithm. 
Our first main contribution is to prove that the running \Mk algorithm converges in some defined sense. The following assumptions are needed throughout the paper.

\begin{assumption}\label{as:tv}
\emph{(Bounded time variations)}
For each time $t_k>0$, there exists a sequence of fixed points $\{\x^*_\tau\}$ from $t_1$ till $t_k$, and a non-negative scalar $\delta$, such that, $\x^*_\tau = \fix \opT_\tau$, for all $\tau \in (1, k]$ and 
\begin{equation}
\|\x^*_{\tau} - \x^*_{\tau-1}\| \leq \delta, \, \forall \tau \in (1,k].
\end{equation} 
%
%
\end{assumption}

Assumption~\ref{as:tv} is a reasonable and mild assumption, which bounds the time variations of the fixed point sets of the time-varying operators. Assumption~\ref{as:tv} is an extended version of the standard required assumption that the Euclidean distance between unique fixed points at subsequent times must be bounded. In fact, if both $\fix\, \opT_{k+1}$ and $\fix \opT_{k}$ are a singleton, then Assumption~\ref{as:tv} coalesces to the standard 
\begin{equation}
\|\x^*_{k+1} - \x^*_k \| \leq \delta, \quad \fix\, \opT_{k+1} = \{\x^*_{k+1}\}, \, \fix\, \opT_{k} = \{\x^*_{k}\}.
\end{equation}



We then consider two additional assumptions on the nature of the operators we are dealing with (these assumptions are not considered to hold simultaneously). 

\begin{assumption} \label{as:bounded}
\emph{(Bounded $\alpha$-averaged operators)}
Let $\{\opT_k\}_{k \in \mathbb{N}_{>0}}$ be a sequence of operators from $\mathbb{R}^n \rightrightarrows \mathbb{R}^n$. We assume that \emph{(i)} each of the $\opT_k$ is an $\alpha_k$-averaged operator in the sense of~\eqref{eq.alphaav}; \emph{(ii)} the image of each operator $\im\, \opT_k = \mathcal{X}_k \subset \mathbb{R}^n$ is a closed compact set, and therefore bounded, and we let $X$ be defined as
\begin{equation}
X : = \max_{k \in \mathbb{N}_{>0}, \x_k \in\mathbb{R}^n} \, \|\opT_k(\x_k)\|.
\end{equation}
\end{assumption}

\begin{assumption}\label{as:contraction}
\emph{(Contractive operators)} Let $\{\opT_k\}_{k \in \mathbb{N}_{>0}}$ be a sequence of operators from $\mathbb{R}^n \rightrightarrows \mathbb{R}^n$. We assume that each $\opT_k$ is a contraction with parameter $L_k\in(0,1)$, in the sense of~\eqref{eq.contr}.
\end{assumption}


Assumptions~\ref{as:bounded}-\ref{as:contraction} will not be considered at the same time. Assumption \ref{as:contraction} is in line with standard literature (which assumes strong smoothness and strong convexity and therefore contractive operators). Assumption \ref{as:bounded} is instead more general and will allow us to generate converging time-varying algorithms for a wider class of optimization problems. Although it may seem restrictive at first sight, many optimization problems verify naturally this assumption. To allow for even more general optimization problem, one would need to remove the boundedness requirement in Assumption~\ref{as:bounded}: as we will argue, this seems to be unavoidable when dealing with $\alpha$-averaged operators, as one needs a measure to quantify the error committed by the time-varying algorithm at each step. A similar requirement is needed in the converging proof of $\epsilon$-(sub)gradient methods, or to quantify errors in regularized problems~\cite{Johansson2008,Nedic2011,Koppel2015}. A similar compactness requirement is imposed in~\cite{Maros2017} for running ADMM algorithms. One alternative approach to substitute this requirement with another one (possibly less restrictive, yet sequence-depending) will be discussed in Section~\ref{sec:alternatives}. A way to enforce this boundedness requirement in a structured way is instead presented in Section \ref{sec:opt.b}.


The following theorem characterizes the convergence and tracking capabilities of Iteration~\eqref{tvfixedpoint}. The proof is given in the appendix. 

\begin{theorem} \emph{(Running \Mk algorithm convergence)}\label{theo:1}
Consider $\{\opT_k\}_{k \in \mathbb{N}_{>0}}$ as a sequence of $\alpha_k$-averaged operators from $\mathbb{R}^n \rightrightarrows \mathbb{R}^n$, and assume $\fix\, T_k \neq \emptyset$, for all $k$. Let $\{\x_k\}_{k \in \mathbb{N}_{>0}}$ be the sequence generated by the running \Mk algorithm \eqref{tvfixedpoint}, for the sequence $\{\opT_k\}_{k \in \mathbb{N}_{>0}}$. Let Assumption~\ref{as:tv} hold. Then, 
\begin{enumerate}
\item[(a)] if Assumption~\ref{as:bounded} holds, the fixed-point residual $\|\opG_k (\x_k) - \x_k\|$ converges in mean to an error bound as, 
\begin{equation}\label{result1}
\frac{1}{T}\,\sum_{k=1}^T\, \alpha_k (1-\alpha_k)\|\opG_k(\x_k) - \x_k\|^2 \leq \frac{1}{T}\|\x_1 - \x^*_1\|^2 + \delta\,(4 X + \delta);
\end{equation}
and, given that $\opT_k = (1-\alpha_k)\opI + \alpha_k \opG_k$,
\begin{equation}\label{result1bis}
\frac{1}{T}\,\sum_{k=1}^T\, \frac{1-\alpha_k}{\alpha_k}\|\opT_k(\x_k) - \x_k\|^2 \leq \frac{1}{T}\|\x_1 - \x^*_1\|^2 + \delta\,(4 X + \delta);
\end{equation}
\item[(b)] if Assumption~\ref{as:contraction} holds, the error norm $\|\x_k - \x_k^*\|$ converges as
\begin{equation}\label{result2}
\|\x_k - \x_k^*\| \leq \hat{L}_k\,\|\x_1 - \x_1^*\| + \frac{1-\bar{L}_k^{k-1}}{1-\bar{L}_k} \delta, 
\end{equation}  
where $\x_k^* \in \fix \opT_k$, $\hat{L}_k = L_1 \cdots L_{k-1}$, and $\bar{L}_k = \max_{k} L_k$. 
\end{enumerate}
\end{theorem}

Theorem~\ref{theo:1} dictates the convergence properties of the running \Mk algorithm. In the case of bounded $\alpha$-averaged operators, we have weak convergence (in fact, in mean) of the fixed-point residual (FPR) error to a neighborhood of the origin. The size of the neighborhood depends on the bound on the time-variations (Assumption~\ref{as:tv}) and on the size of the image of the operators. If we defined as $\underline{a} := \min_k (1-\alpha_k)/\alpha_k$, then the mean fixed-point residual error approaches asymptotically:
\begin{equation}
\limsup_{T\to\infty} \frac{1}{T}\,\sum_{k=1}^T\, \|\opT_k(\x_k) - \x_k\|^2 = \frac{\delta}{\underline{a}}\,{(4 X + \delta)}.
\end{equation}
When $\delta = 0$, we re-obtain the same results of the time-invariant case. If the operators are contractive, then a better error norm can be proven to be converging. In particular, we have strong convergence of the decision variables $\x_k$ to a fixed point of the operators up to a bound due to the time variations. In the limit, 
\begin{equation}
\limsup_{k\to \infty} \|\x_k - \x^*_k\| = \delta/(1-\bar{L}_k).
\end{equation}

We will study the consequences of this theorem for convex optimization in Section~\ref{sec:opt} and subsequent ones, while in the next section we propose a way to substitute the boundedness requirement of Assumption~\ref{as:bounded}.


\section{Alternative characterization: ``practical'' convergence}\label{sec:alternatives}

Consider Assumption~\ref{as:bounded}. As one can appreciate from the proof of Theorem~\ref{theo:1}, namely~\eqref{dummy:69}, this requirement is needed to lower bound the inner product,
\begin{equation}\label{inner}
(\x_{k+1} - \x^*_{k})^\transp(\x^*_{k+1}- \x^*_k). 
\end{equation}
Another, in fact related, road that can be taken to bound~\eqref{inner} is to bound the variation of the square distances, as encoded in the following Assumption.
\begin{assumption} \label{as:stv}
\emph{(Squared time-variations)}
Let $\{\x_k\}_{k \in \mathbb{N}_{>0}}$ be the sequence generated by the running \Mk algorithm \eqref{tvfixedpoint}. For each time $t_k>0$, there exists a sequence of fixed points $\{\x^*_\tau\}$ from $t_1$ till $t_k$,  and a non-negative scalar $d$, such that, $\x^*_\tau = \fix \opT_\tau$, for all $\tau \in (1, k]$ and 
\begin{equation}\label{horder}
\|\x_{\tau+1} -\x^*_{\tau+1}\|^2 \leq  \|\x_{\tau+1} - \x^*_{\tau}\|^2 + d^2, \, \forall \tau \in (1, k].
\end{equation} 
\end{assumption}

Assumption~\ref{as:stv} (despite being depended on the sequence $\{\x_k\}_{k \in \mathbb{N}_{>0}}$) is a reasonable assumption in many practical situations, e.g., when the optimizer set is bounded and $\x_{\tau+1}$ is not far-away from the optimizer trajectory. In this context, the results that rely on this assumption will be called ``practical'' convergence result.  

By developing the squares, one arrives at a lower boundedness condition on the inner product~\eqref{inner}, so Assumption~\ref{as:stv} de-facto enforces Assumption~\ref{as:bounded}. Requirement~\eqref{horder} can be interpreted also as a bound on the variations of the fixed point sets. We know that if the sets are invariant, then~\eqref{horder} must hold with $d = 0$. When they vary, we need to require that~\eqref{horder} holds. 



\begin{theorem} \emph{(Time-varying \Mk algorithm ``practical'' convergence)}\label{theo:3}
Let $\{\opT_k\}_{k \in \mathbb{N}_{>0}}$ be a sequence of $\alpha_k$-averaged operators from $\mathbb{R}^n \rightrightarrows \mathbb{R}^n$, and assume $\fix\, T_k \neq \emptyset$, for all $k$. Let $\{\x_k\}_{k \in \mathbb{N}_{>0}}$ be the sequence generated by the running \Mk algorithm \eqref{tvfixedpoint}, for the sequence $\{\opT_k\}_{k \in \mathbb{N}_{>0}}$. Let Assumption~\ref{as:stv} hold. 
Then, the error norm $\|\opG_k (\x_k) - \x_k\|$ converges in mean to an error bound as, 
\begin{equation}
\frac{1}{T}\,\sum_{k=1}^T\, \alpha_k (1-\alpha_k)\|\opG_k (\x_k) - \x_k\|^2 \leq \frac{1}{T}{\|\x_1 - \x^*_1\|^2} + d^2.
\end{equation}
\end{theorem}
 
Theorem~\ref{theo:3} offers a different view-point from Theorem~\ref{theo:1}, case \emph{(a)}. In both cases convergence of a weighted version of the fixed point residual error goes as $O(1/T)$ up to a bounded error depending on the variability of the fixed points in time.


\section{Asymptotically vanishing ``errors''}\label{sec:vanishing}

In this section, we briefly consider the case in which the operator $\opT(t)$ changes in time, but eventually reaches some steady-state operator $\bar{\opT}$. Although in this paper we are more interested in tracking properties, cases for which the operator reaches a steady-state can be relevant from an application perspective, for example in the online convex optimization framework~\cite{Shalev-Shwartz2012}.

\begin{corollary} \emph{(Running \Mk algorithm convergence for vanishing errors)}\label{theo:5}
Consider point (a) of Theorem~\ref{theo:1} and the modified Assumption~\ref{as:tv} where $\delta$ is now a time-dependent quantity $\delta_k$. If, 
\begin{equation}\label{as.delta}
\lim_{k\to \infty} \sum_{i=1}^k \delta_i < \infty,
\end{equation}
then the fixed-point residual converges strongly to zero, e.g., $\|\opG_k (\x_k) - \x_k\| \to 0$ as $k \to \infty$. For point (b) of Theorem~\ref{theo:1}, if~\eqref{as.delta} holds, then $\|\x_k - \x_k^*\| \to 0$.
\end{corollary}
\begin{proof}
Direct by Theorem~\ref{theo:1}, since~\eqref{as.delta} implies also $\lim_{k\to \infty} \sum_{i=1}^k \delta_i^2 < \infty$.
\qed
\end{proof}
 

\section{Connections with existing work}\label{sec:literature}

\Mk iteration has been studied extensively and we do not have the ambition here to give an exhaustive account of all the results that have appeared, since the main aim of this paper is its connection with time-varying optimization (rather than an improvement of the \Mk iteration itself). However, it is relevant to briefly report connections with existing works in relation to the time-varying version of the \Mk iteration.

For a general \Mk iteration account, besides the standard reference~\cite{Bauschke2011}, interested reader could also find a complete analysis of \Mk iteration with various error conditions in~\cite{Combettes2002}, which also provide the notion of convergence with overrelaxed parameters, i.e., $\alpha \in (0,2)$. Tighter bounds are provided in~\cite{Davis2014}, while for recent results and surveys on splitting methods see~\cite{Eckstein1989, Ryu2015, Frankel2015, Bianchi2016, Garrigos2017}. Rather recently, various linear convergence results similar to~\eqref{linear} have appeared without requiring the operator to be a contraction~\cite{Bauschke2015,Banjac2016}. In particular, Result~\eqref{linear} holds iff the $\alpha$-averaged operator $\opT$ is linearly bounded, i.e., 
\begin{equation}
\|\x - \x^*\| \leq \kappa \|\x - \opT(\x)\|, \quad \forall \x\in\mathbb{R}^n, \, \kappa\geq0. 
\end{equation}
A contraction is a linearly bounded operator but not vice versa (and, in practice, it is not completely straightforward to make sure that an operator is linearly bounded, besides the case in which is a contraction). This convergence result has strong links with recent relaxed versions of strong convexity~\cite{Necoara2015}.
Other regularity assumptions to allow for decision variable convergence are explored in~\cite{Borwein2015}, while acceleration of the \Mk iteration to super-linear convergence is explored in~\cite{Themelis2016}. Most of the aforementioned works frame their contributions in Hilbert and even Banach spaces, while here (for simplicity) we restrict ourselves to $\mathbb{R}^n$. 

When one is concern with convergence of the sequence $\{
\x_k\}_{k \in \mathbb{N}_{>0}}$ towards a fixed point of the time-invariant $\opT$ (or equivalently $\opG$), one would like to establish the strong convergence of the residuals $\|\x_k - \opG(\x_k)\| \to 0$, a property referred to as asymptotic regularity. An explicit estimate for the residual is available~\cite{Cominetti2014, Bravo2016}, as
\begin{equation}
\|\opG(\x_k)-\x_k\| \leq \frac{\textrm{diam}(\mathcal{X})}{\sqrt{\pi k \alpha (1-\alpha)}},
\end{equation}
where $\mathcal{X}$ is the bounded image set of $\opG$, while the constant $1/\sqrt{\pi}$ is tight. 

If one allows for errors in the computation of the operator, then could consider the inexact \Mk iteration as 
\begin{equation}\label{ikm}
\x_{k+1} = \opT(\x_k) + \alpha\e_k,
\end{equation}
where $\e_k \in \mathbb{R}^n$ is an error vector, supposed bounded as $\|\e_k\|\leq \epsilon_k$. A variety of results have appeared to characterize convergence of~\eqref{ikm}, see for example~\cite{Combettes2002,Bravo2017} and reference therein. The main point in the aforementioned work is that the fixed point set is time-invariant but we commit errors at every discrete time steps. Then, under rather mild assumptions (which are verified if the operator has a bounded image set) one can show that, 
\begin{equation}\label{ikm.result}
\|\opG(\x_k)-\x_k\|\leq \Omega(\kappa,\alpha, \{\epsilon_i\}_{i \in \{1,\dots,k+1\}}),
\end{equation}
where $\Omega$ is a function of $\alpha$, on a bound on $\mathcal{X}$ here indicated as $\kappa$, and the error vectors bounds $\epsilon_i$, and it is bounded, whenever $\epsilon_i$ is bounded. If $\sum_{i}{\epsilon_i}< \infty$ then $\|\opG(\x_k)-\x_{k} \| \to 0$. These results are similar to the ones in Corollary~\ref{theo:5}.

Diagonal \Mk iterations~\cite{Zhao2005,Xu2006,Peypouquet2009} have also appeared -- where at each iteration one purposely chooses to use a different, perhaps easier to compute, operator $\opT_{k}$ --  and as indicated by~\cite{Bravo2017}, these methods can be interpreted as inexact iterations where $\alpha \e_k = \opT_{k}(\x_k) - \opT(\x_k)$. 

To the best of the author's knowledge, no inexact method have been appeared to tackle the case in which the operator $\opT$ and its fixed point set is time-varying, as we study here. 

\Mk iterations have strong connections with evolution equations of the form
\begin{equation}\label{ode}
\frac{\textrm{d} u(t)}{\textrm{d} t} + (\opI - \opG) u(t)  = f(t), \quad u(0) = x_0.
\end{equation}
In fact, by discretizing~\eqref{ode} with a forward-Euler method of fixed time period $h<1$, one obtain the recursion
\begin{equation}
\frac{u_{k+1}-u_k}{h} + (\opI - \opG) u_k  = f_k, \quad u_0 = x_0,
\end{equation}
or
\begin{equation}\label{mk_ode}
u_{k+1} = ((1-h)\opI  + h \opG) u_k  - h f_k, \quad u_0 = x_0,
\end{equation}
which is an inexact \Mk iteration, whenever $\opG$ is nonexpansive. Characterizations of convergence of~\eqref{mk_ode} when $f_k$ is bounded and asymptotically vanishing are also appeared in the literature, e.g.,~\cite{Bravo2017}. A survey of some recent results and connections between continuous and discrete case can be found in~\cite{Peypouquet2010}, which mainly focus on monotone operators and existence of solutions. 

Time-varying fixed point sets and operators, as in our case, can be derived instead from the evolution equation
\begin{equation}\label{ode_2}
\frac{\textrm{d} u(t)}{\textrm{d} t} + (\opI - \opG(t)) u(t)  = 0, \quad u(0) = x_0.
\end{equation}
which has been studied considerably less in the context of \Mk iterations. A pioneer work is the one by J.~J. Moreau~\cite{Moreau1977}, which studies a particular~\eqref{ode_2} in the context of moving convex sets and he proposes a running \Mk algorithm in the line of~\eqref{tvfixedpoint}, which he names \emph{catching-up} algorithm (and whose error w.r.t. the continuous solution is proven bounded if the sampling period is bounded). More recently, the results in~\cite{Briceno-Arias2016} offers a broader perspective on equations (and differential inclusions~\cite{Dontchev1992}) of the type of~\eqref{ode_2}, when $(\opI - \opG(t))$ is maximally monotone. The focus is again on the property of the continuous solution and its (consistent) discrete approximation. Finally, the works~\cite{Cojocaru2005,Nagurney2006} discuss evolution variational inequalities (EVI) and a discrete running algorithm in the line of~\eqref{tvfixedpoint} is presented, whose convergence is proven under strong monotonicity assumptions.  We feel that promising future research directions lie in the line of research put forth by~\cite{Cojocaru2005,Nagurney2006,Briceno-Arias2016}.

\section{Consequences for Time-varying Convex Optimization}\label{sec:opt}

A number of corollaries can now be derived based on the result of Theorem~\ref{theo:1} (we will not consider Theorem~\ref{theo:3} here, yet its application would be direct), which are summarized in Table~\ref{tab.1}. In Table~\ref{tab.2}, we report additional results in terms of objective convergence, which will be obtained in Section~\ref{sec:obj_conv}. 

In order to prove some of the results, we need the following standard lemma, reported here for simplicity. 

\begin{lemma}\label{lemma:scsm}\cite{Rockafellar1970, Taylor2017}
Let $f:\mathbb{R}^n \to \mathbb{R}$ be a CCP function. Let $\mathcal{X}\subseteq \mathbb{R}^n$ be a closed convex set. Let $i_{\mathcal{X}}(\x)$ be the indicator function, which is $0$ for $\x \in \mathcal{X}$ and $+\infty$ otherwise. Consider the extended valued function $f_\mathrm{e}:\mathbb{R}^n \to \mathbb{R}\cup\{\infty\} = f + i_{\mathcal{X}}$. For $f_\mathrm{e}$ the following facts are true. 
\begin{enumerate}
\item[\emph{i)}] The subgradient operators of $f_e$ and of its conjugate are reciprocal of each other: $\partial f_e^{-1} = \partial \cf_e$;
\item[\emph{ii)}] If $f$ is strongly convex with constant $m$ over $\mathbb{R}^n$, then $\cf_e$ is strongly smooth with constant $1/m$.  
\item[\emph{iii)}] If $f$ is strongly smooth with constant $M$ over $\mathbb{R}^n$, and $\mathcal{X} = \mathbb{R}^n$, then $\cf_e$ is strongly convex with constant $1/M$.  
\end{enumerate}
\end{lemma}


\begin{table}
\caption{Results presented in Sections~\ref{sec:opt}-\ref{sec:opt.b} based on Theorem~\ref{theo:1}. The same results hold for Theorem~\ref{theo:3} by substituting all $\subset$ with $\subseteq$ and without the need for the bounding set $\mathcal{B}$.}
\centering
\begin{tabular}{ccc}
\toprule
Method & Corollary & Result for FPR and variable convergence\\ \toprule
Proj. gradient & \ref{co:gradient}(a)& $f$ strongly smooth, $\mathcal{X}_k \subset \mathbb{R}^n$ \\ 
& \ref{co:gradient}(b)& $f$ strongly smooth and strongly convex, $\mathcal{X}_k \subseteq \mathbb{R}^n$\\ \hline
Proximal point & \ref{co:pp}(a) & $f$ CCP, $\mathcal{X}_k \subset \mathbb{R}^n$ \\
& \ref{co:pp}(b)  & $f$ strongly convex, $\mathcal{X}_k \subseteq \mathbb{R}^n$ \\ \hline
F-B splitting & \ref{co:fb}(a)& $g$ CCP, $f$ strongly smooth, $\mathcal{X}_k \subset \mathbb{R}^n$ \\ 
& \ref{co:fb}(b)& \hskip-.5cm$g$ CCP, $f$ strongly smooth and strongly convex, $\mathcal{X}_k \subseteq \mathbb{R}^n$ \\\hline
Dual ascent ineq.  & \ref{co:da}(a)& $f$ strongly convex, $\mathcal{X} \subseteq \mathbb{R}^n$ \\
& \ref{co:da}(b)& \hskip-.5cm$f$ strongly smooth and strongly convex, $\mathcal{X}  = \mathbb{R}^n$, $\sigma_{\min} > 0$\\ \hline
Dual ascent eq. 1  & Eq. \eqref{tvdual2} & Same as dual ascent ineq. with $\p \in \mathcal{B}$ \\ \hline
Dual ascent eq. 2 & \ref{co:daeq} & \hskip-.5cm$f$ strongly smooth and strongly convex, $\mathcal{X}  = \mathbb{R}^n$
 \\ \hline
D-R splitting & \ref{co:dr}(a)& $f,g$ CCP, $\mathcal{X}_k \subseteq \mathbb{R}^n$,  $\z \in \mathcal{B}$  \\  & \ref{co:dr}(b)& \hskip-.5cm$g$ CCP, $f$ strongly smooth and strongly convex, $\mathcal{X}_k = \mathbb{R}^n$ \\ \hline
ADMM & Eq. \eqref{tvadmm}& $f,g$ CCP, $\mathcal{X}_k \subseteq \mathbb{R}^n$,  $\p \in \mathcal{B}$  \\  & \ref{co:admm}& \hskip-.5cm$g$ CCP, $f$ str. smooth and str. convex, $\mathcal{X}_k = \mathbb{R}^n$, $\sigma_{\min} > 0$ \\ \bottomrule
\end{tabular}
\label{tab.1}
\end{table}



\begin{table}
\caption{Results presented in Section~\ref{sec:obj_conv} based on Theorem~\ref{theo:1} and Lemma~\ref{lemma.joint}.}
\centering
\begin{tabular}{cccc}
\toprule
Method & Result & & Result for objective convergence\\ \toprule
Proj. gradient & Corollary~\ref{co:gr-obj}& & $f$ strongly smooth, $\mathcal{X}_k = \mathcal{X} \subset \mathbb{R}^n$ \\ \hline
Proximal point & Corollary~\ref{co:pp-obj} & & $f$ CCP, $\mathcal{X}_k = \mathcal{X} \subset \mathbb{R}^n$ \\ \hline
F-B splitting & Proposition~\ref{co:fb_obj}& & $g$ CCP, $f$ strongly smooth, $\mathcal{X}_k = \mathcal{X} \subset \mathbb{R}^n$ \\ \bottomrule
\end{tabular}
\label{tab.2}
\end{table}


\subsection{Gradient method}

First of all, consider the time-varying convex optimization problem
\begin{equation}
\minimize_{\x \in \mathcal{X}(t)}\, f(\x; t)
\end{equation}
with uniformly CCP function $f: \mathbb{R}^n\times \mathbb{R}_{+} \to \mathbb{R}$, and uniformly convex set $\mathcal{X}(t) \subseteq \mathbb{R}^n$. Sample the problem for sampling times $t_k, k =1,2,\dots$ and solve the equivalent
\begin{equation}\label{zeros}
\minimize_{\x \in \mathcal{X}_k}\, f(\x; t_k) = f_k(\x) \quad \iff \quad \textrm{find} \,\, \zer (\partial f_k + N_{\mathcal{X}_k}),
\end{equation} 
where $N_{\mathcal{X}_k}$ is the normal cone operator for $\mathcal{X}_k$. Finding the zeroes of the operator on the right is equivalent of finding the fixed points of the composition: 
\begin{equation}
\Pi_{\mathcal{X}_k}(\opI - \lambda \partial f_k)
\end{equation}
where $\Pi_{\mathcal{X}_k}$ is the projection operator~\cite{Eckstein1989}. It is not difficult to see that the running projected gradient algorithm, 
\begin{enumerate}
\item[1.] Set $\x_1$ arbitrarily,
\item[2.] \textbf{for} $k>1$ \textbf{do}:

Compute the next approximate fixed point
\begin{equation}\label{tvprojgradient}
\x_{k+1} = \Pi_{\mathcal{X}_k}(\opI - \lambda \partial f_k)
(\x_k)
\end{equation}
\end{enumerate}
is a special case of the running \Mk algorithm. 

\begin{corollary}\label{co:gradient}
Consider the running projected gradient defined in~\eqref{tvprojgradient} and the generated sequence $\{\x_k\}_{k \in \mathbb{N}_{>0}}$. Let $f_k$ be differentiable for all $k$ (i.e., $\partial f_k = \nabla f_k$) and be strongly smooth with constant $M_k\leq M$ over $\mathbb{R}^n$. Fix $\lambda \in (0, 2/M)$ so that the operator $\Pi_{\mathcal{X}_k}(\opI - \lambda \nabla f_k)$ is $\alpha$-averaged with $\alpha= \alpha_k$. Let Assumption~\ref{as:tv} hold.
\begin{enumerate}
\item[(a)] Let the sets $\mathcal{X}_k$ be compact, and therefore bounded, and let $X$ be defined as $X = \max_{k} \|\mathcal{X}_k\|$.  Define $\underline{a}:= \min_k (1-\alpha_k)/\alpha_k$. Then, the sequence $\{\x_k\}_{k \in \mathbb{N}_{>0}}$ converges in the sense of~\eqref{result1} and in particular,
\begin{equation}\label{result1.1}
\frac{1}{T}\,\sum_{k=1}^T\, \|\Pi_{\mathcal{X}_k}(\x_k - \lambda \nabla f_k(\x_k))   - \x_k\|^2 \leq \frac{1}{\underline{a} T}{\|\x_1 - \x^*_1\|^2} + \frac{\delta}{\underline{a}} \,(4 X + \delta );
\end{equation} 
\item[(b)] Let the functions $f_k$ be strongly convex over $\mathbb{R}^n$, for all $k$, with constant $m_k$. Then, the operator $\Pi_{\mathcal{X}_k}(\opI - \lambda \nabla f_k)$ is a contraction with $L_k = \max\{|1-\lambda m_k|,|1-\lambda M_k|\}$ and we obtain primal convergence in the sense of~\eqref{result2}: 
\begin{equation}\label{result1.2}
\|\x_k - \x_k^*\| \leq \hat{L}_{k}\,\|\x_1 - \x_1^*\| + \frac{1-\bar{L}^{k-1}_k}{1-\bar{L}_k} \delta, \quad L_k = \max\{|1-\lambda m_k|,|1-\lambda M_k|\}.
\end{equation}
\end{enumerate}
\end{corollary}

\begin{proof}
Case \emph{(a)}. The sequence of operators $\{\opT_k\} = \{\Pi_{\mathcal{X}_k}(\opI - \lambda \nabla f_k)\}$ verifies Assumption~\ref{as:bounded}. With $\lambda \in (0,2/M)$, the operator $\opI - \lambda \nabla f_k$ is $\alpha$-averaged with $\alpha = \lambda M_k/2$, while the projection operator is $\alpha$-averaged with $\alpha = 1/2$, \cite{Ryu2015}. Therefore their composition is $\alpha$-averaged according to Proposition~\ref{composition} with $\alpha$ constant equal to $\alpha_k = 1/(2-\lambda M_k/2)$. By applying~\eqref{result1bis}, result~\eqref{result1.1} follows.

Case \emph{(b)}. Contraction of the operator $\{\Pi_{\mathcal{X}_k}(\opI - \lambda \nabla f_k)\}$ follows from the fact that, when $f(\x; t_k)$ is strongly convex and strongly smooth, then $\opI - \lambda \nabla f_k$ is a contraction for $\lambda \in (0, 2/M)$, \cite{Ryu2015}, and in particular the contraction factor is $L_k = \max\{|1-\lambda m_k|,|1-\lambda M_k|\}$. Composition of a contraction and a non-expansive operator $\Pi_{\mathcal{X}_k}$ is still a contraction with the same contraction factor. Therefore the sequence of operators $\{\Pi_{\mathcal{X}_k}(\opI - \lambda \nabla f_k)\}$ verifies Assumption~\ref{as:contraction} and result~\eqref{result1.2} follows.
\qed
\end{proof}

\vskip2mm

The fixed-point residual $\|\Pi_{\mathcal{X}_k}(\x_k - \lambda \nabla f_k(\x_k))   - \x_k\|$ in result~\eqref{result1.1} is also known as global error estimate~\cite{Necoara2015}. Note that preliminary results for convergence of time-varying projected gradient have appeared in~\cite{Simonetto2014c} albeit with a different analysis technique. 

Before proceeding, it is interesting to revisit the need for Assumption~\ref{as:bounded} (and equivalently Assumption~\ref{as:stv}) by using standard proof tools of the gradient method. For strongly smooth functions, in the time-invariant setting $f(\x)$ is a Lyapunov function for the gradient method. In fact, it is $f(\x)\geq 0$ for all $\x$ w.l.g., and for strong smoothness:
\begin{align}
f(\x_{k+1}) &= f(\x_{k}- \alpha \nabla f(\x_k)) \leq f(\x_{k}) - \alpha \|\nabla f(\x_k)\|^2 + \frac{\alpha M}{2}\|\nabla f(\x_k)\|^2 = \nonumber \\
 &\leq f(\x_k) - \alpha(1-\alpha M/2) \|\nabla f(\x_k)\|^2, \label{eq.gradlyap}
\end{align}
where $M$ is the strong smoothness constant and $\alpha < 2/M$ is the stepsize. By~\eqref{eq.gradlyap}, the gradient method converges to the optimum. When the function $f$ is time-varying then,
\begin{equation}
f_k(\x_{k+1}) \leq f_k(\x_k) - \alpha(1-\alpha M/2) \|\nabla f_k(\x_k)\|^2,
\end{equation}
and therefore
\begin{equation}
f_{k+1}(\x_{k+1}) \leq f_k(\x_k) - \alpha(1-\alpha M/2) \|\nabla f_k(\x_k)\|^2 + (f_{k+1}(\x_{k+1})-f_{k}(\x_{k+1})).
\end{equation}
Assuming w.l.g. that the optimum of $f_{k+1}$ is the same as the optimum of $f_k$ and they are both $0$, the value for $f_{k+1}(\x_{k+1})$ can be at most $M/2 \|\x_{k+1} - \x^*_{k+1}\|^2$ and the one for $f_{k}(\x_{k+1})$ can be at most $M/2 \|\x_{k+1} - \x^*_{k}\|^2$. If one looks at the minimax error for every iterations, one has 
\begin{equation}
\min_{f_{k}} \max_{f_{k+1}} f_{k+1}(\x_{k+1})-f_{k}(\x_{k+1}) = M/2 (\|\x_{k+1} - \x^*_{k+1}\|^2 - \|\x_{k+1} - \x^*_{k}\|^2).
\end{equation} 
The error term is exactly $d^2\,M/2$ in~\eqref{horder}, or requires Assumption~\ref{as:bounded} to be bounded.

\subsection{Proximal-point method}

Another equivalent rewriting of finding the zeroes of the operator in the rightmost term of~\eqref{zeros} is finding the set
\begin{equation}\label{proxdef}
\fix\, (I+\lambda (\partial f_k + N_{\mathcal{X}_k}))^{-1} =: \fix \, \opR_k. 
\end{equation}
The operator $\opR_k$ is the called the resolvent of the operator $\partial f_k + N_{\mathcal{X}}$. Since the latter is a maximal monotone operator for all CCP functions $f_k$, then $\opR_k$ is an $\alpha$-averaged operator. The resulting running algorithm is the running proximal point method, 
\begin{enumerate}
\item[1.] Set $\x_1$ arbitrarily,
\item[2.] \textbf{for} $k>1$ \textbf{do}:

Compute the next approximate fixed point
\begin{equation}\label{tvprox}
\x_{k+1} = \opR_k
(\x_k) = \arg\min_{x \in \mathcal{X}_k} \left\{f_k(\x) + \frac{1}{2 \lambda}\|\x - \x_{k}\|^2\right\} = \prox_{f_k,\mathcal{X}_k,\lambda}(\x_k),
\end{equation}
\end{enumerate}
whose convergence goes as follows. 

\begin{corollary}\label{co:pp}
Consider the running proximal point defined in~\eqref{tvprox} and the generated sequence $\{\x_k\}_{k \in \mathbb{N}_{>0}}$. Let Assumption~\ref{as:tv} hold.
\begin{enumerate}
\item[(a)] Let the sets $\mathcal{X}_k$ be compact, and therefore bounded, and let $X$ be defined as $X = \max_{k} \|\mathcal{X}_k\|$.  Then, the sequence $\{\x_k\}_{k \in \mathbb{N}_{>0}}$ converges in the sense of~\eqref{result1} and in particular,
\begin{equation}\label{result2.1}
\frac{1}{T}\,\sum_{k=1}^T\, \|\prox_{f_k,\mathcal{X}_k,\lambda}(\x_k) 
- \x_k\|^2 \leq \frac{1}{T}{\|\x_1 - \x^*_1\|^2} +  \delta\,(4 X + \delta);
\end{equation} 
\item[(b)] Let the functions $f_k$ be strongly convex over $\mathbb{R}^n$, for all $k$, with constant $m_k$. Then the operator $\opR_k$ is a contraction with $L_k = (1+m_k\lambda)^{-1}$ and we obtain primal convergence in the sense of~\eqref{result2}: 
\begin{equation}\label{result2.2}
\|\x_k - \x_k^*\| \leq \hat{L}_{k}\,\|\x_1 - \x_1^*\| + \frac{1-\bar{L}^{k-1}_k}{1-\bar{L}_k} \delta, \quad L_k = (1+m_k\lambda)^{-1}.
\end{equation}
\end{enumerate}
\end{corollary}

\begin{proof}
Case \emph{(a)}. The sequence of operators $\opR_k$ verifies Assumption~\ref{as:bounded}, since $\mathcal{X}_k$ is compact and $\opR_k$ is the resolvent of a maximal monotone operator. The operator $\opR_k$ is also $\alpha$-averaged with $\alpha = 1/2$, \cite{Ryu2015}, from which result~\eqref{result2.1} follows.

Case \emph{(b)}. Contraction of the operator $\opR_k$ under strong convexity follows from~\cite[Eq.s~(1.14)-(1.15)]{Rockafellar1976} (or equivalently from Lemma~\ref{lemma:scsm} and the definition of $\opR_k$ in~\eqref{proxdef}: From~\eqref{proxdef}, one notices that $\opR_k  = \partial F_e^{-1} = \partial F^{\star}_e$, for the function $F(\x) = \|\x\|^2/2 + \lambda f_k(\x)$, and then uses the fact that $F(\x)$ is strongly convex with constant $1+m_k \lambda$ and Lemma~\ref{lemma:scsm} to conclude. Therefore the sequence of operators $\opR_k$ verifies Assumption~\ref{as:contraction} and result~\eqref{result2.2} follows. 
\qed
\end{proof}

\vskip2mm

Note that the convergence requirements of the running proximal point method are less restrictive than the running gradient method, as it happens in the time-invariant case. In particular, in the case of running proximal point method the functions $f(\x;t_k)$ do not have to be strongly smooth, and the stepsize $\lambda$ can be picked arbitrarily. 

\subsection{Forward-backward splitting for composite optimization}

Consider the composite optimization problem
\begin{equation}\label{composite}
\minimize_{\x \in \mathcal{X}(t)} \, f(\x; t) + g(\x;t), 
\end{equation}
where $f,g: \mathbb{R}^n\times \mathbb{R}_{+} \to \mathbb{R}$ are uniformly a CCP function, while the convex set $\mathcal{X}(t) \subseteq \mathbb{R}^n$ uniformly. Sample the optimization problem at $t_k$, $k = 1, 2, \dots$ and consider the following running version of the celebrated forward-backward splitting method:
\begin{enumerate}
\item[1.] Set $\x_1$ arbitrarily,
\item[2.] \textbf{for} $k>1$ \textbf{do}:

Compute the next approximate fixed point
\begin{equation}\label{tvfb}
\x_{k+1} = \prox_{g_k,\mathcal{X}_k, \lambda}(\x_k - \lambda \partial f_k(\x_k)).
\end{equation}
\end{enumerate}

\begin{corollary}\label{co:fb}
Consider the running forward-backward splitting defined in~\eqref{tvfb} and the generated sequence $\{\x_k\}_{k \in \mathbb{N}_{>0}}$. Let $f_k$ be strongly smooth for all $k$ with constant $M_k\leq M$ (which implies $\partial f_k = \nabla f_k$).  Fix $\lambda \in (0, 2/M)$ so that the operator $\opI - \lambda \nabla f_k$ is $\alpha$-averaged. Let Assumption~\ref{as:tv} hold.
\begin{enumerate}
\item[(a)] Let the sets $\mathcal{X}_k$ be compact, and therefore bounded, and let $X$ be defined as $X = \max_{k} \|\mathcal{X}_k\|$.  Define $\underline{a}:= \min_k (1-\alpha_k)/\alpha_k$, where $\alpha_k = 1/(2-\lambda M_k/2)$. Then, the sequence $\{\x_k\}_{k \in \mathbb{N}_{>0}}$ converges in the sense of~\eqref{result1} and in particular,
\begin{equation}\label{result2.1bis}
\frac{1}{T}\,\sum_{k=1}^T\, \|\prox_{g_k,\mathcal{X}_k, \lambda}(\x_k - \lambda \nabla f_k(\x_k))   - \x_k\|^2 \leq \frac{1}{\underline{a} T}{\|\x_1 - \x^*_1\|^2} + \frac{\delta}{\underline{a}}\,(4 X + \delta);
\end{equation} 
\item[(b)] Let the functions $f_k$ be strongly convex over $\mathbb{R}^n$, for all $k$, with constant $m_k$. Then the operator $\opI - \lambda \nabla f_k$ is a contraction with $L_k = \max\{|1-\lambda m_k|,|1-\lambda M_k|\}$ and we obtain primal convergence in the sense of~\eqref{result2}: 
\begin{equation}\label{result2.2bis}
\|\x_k - \x_k^*\| \leq \hat{L}_{k}\,\|\x_1 - \x_1^*\| + \frac{1-\bar{L}^{k-1}_{k}}{1-\bar{L}_{k}} \delta, \quad L_k = \max\{|1-\lambda m_k|,|1-\lambda M_k|\}.
\end{equation}
\end{enumerate}
\end{corollary}

\begin{proof}
The proof is similar to the one of Corollary~\ref{co:gradient}. The operator described in~\eqref{tvfb} is the composition of a proximal operator and $\opI - \lambda \nabla f_k$. The proximal operator is $\alpha$-averaged, with $\alpha = 1/2$, and therefore the composition is $\alpha$ averaged with $\alpha_k = 1/(2-\lambda M_k/2)$. By~\eqref{result1bis}, result~\eqref{result2.1bis} follows. Result~\eqref{result2.2bis} can be proven by noticing that the proximal operator is non-expansive and therefore if $f_k$ is also strongly convex, $\opI - \lambda \nabla f_k$ and thus the whole operator in~\eqref{tvfb} becomes a contraction. 
\qed
\end{proof}

\subsection{Dual ascent with inequality constraints}

We look now at the linearly constrained optimization problem

\begin{equation}\label{inequality}
\minimize_{\x \in \mathcal{X}} \, f(\x; t), \, \textrm{subject to} \, \A \x \leq \b
\end{equation}
where $f: \mathbb{R}^n\times \mathbb{R}_{+} \to \mathbb{R}$ is uniformly a CCP function, while $\A \in \mathbb{R}^{m\times n}$ and $\b \in \mathbb{R}^m$. The inequality is intended element-wise and the convex set $\mathcal{X} \subseteq \mathbb{R}^n$. We sample the optimization problem at $t_k$, $k = 1, 2, \dots$ and we assume that strong duality holds for each $t_k$. The dual problems of the sampled primal ones are
\begin{equation}\label{eq.dual}
\maximize_{\p \in \mathbb{R}^m_{+}} \, q_k(\p),
\end{equation}
where $\p$ is the vector collecting the dual variables associated with the inequality constraint $\A\x \leq \b$. Under Slater's condition, the optimal dual variables are bounded~\cite{Nedic2009a}, i.e.,
\begin{equation}
\p_k^* \in \left\{\p_k \in \mathbb{R}^m_{+}\,|\, \|\p_k\|\leq \frac{1}{\gamma}(f_k(\bar{\x}) - q_k(\tilde{\p}))\right\} =: \mathcal{P}_k,
\end{equation}
where $\gamma = \min_{1\leq i \leq m} [\b - \A\bar{\x}]_i$, while $\bar{\x}$ is a Slater's vector and $\tilde{\p}$ is any dual feasible variable. We note that $\mathcal{P}_k$ can be computed easily online, since $\bar{\x}$ is constant. 

The running Lagrangian dual ascent scheme to find and track the time-varying dual optimal variables of~\eqref{eq.dual} is the following recursion,
\begin{enumerate}
\item[1.] Set $\p_1$ arbitrarily,
\item[2.] \textbf{for} $k>1$ \textbf{do}:

Compute the next approximate fixed point
\begin{equation}\label{tvdual}
\p_{k+1} = \Pi_{\mathcal{P}_k}(\opI + \lambda \partial q_k)(\p_k) = \left\{\begin{array}{l}\x_{k+1} \in \arg\min_{\x \in \mathcal{X}} \{f_k(\x) + \p_k^\transp (\A\x - \b) \} \\ \p_k = \Pi_{\mathcal{P}_k}(\p_k + \lambda (\A\x_{k+1} - \b))\, .  \end{array}\right.
\end{equation}
\end{enumerate}

\begin{corollary}\label{co:da}
Consider the running Lagrangian dual ascent defined in~\eqref{tvdual} and the generated sequence $\{\x_k\}_{k \in \mathbb{N}_{>0}}$. Let the maximum singular value of $\A$ be $\sigma_{\max}$, and assume it is positive w.l.g.\,. Let $f_k$ be strongly convex for all $k$ with constant $m_k$ (which implies the dual function being differentiable, $\partial q_k = \nabla q_k$, and $-q_k$ being strongly smooth, with constant $\sigma^2_{\max}/m_k$ ).  Fix $\lambda \in (0, 2 m/\sigma^2_{\max})$ so that the operator $\Pi_{\mathcal{P}_k}(\opI + \lambda \nabla q_k)$ is $\alpha$-averaged with $\alpha=\alpha_k$. Let Assumption~\ref{as:tv} hold.
\begin{enumerate}
\item[(a)] Let $\delta$ measure the variability of the optimal dual variables, while let $\mathcal{X}_k$ in Assumption~\ref{as:bounded} be $\mathcal{X}_k = \mathcal{P}_k$. Define $\underline{a}:= \min_k (1-\alpha_k)/\alpha_k$. Then, the sequence $\{\p_k\}_{k \in \mathbb{N}_{>0}}$ converges in the sense of~\eqref{result1} and in particular,
\begin{equation}\label{result3.1}
\frac{1}{T}\,\sum_{k=1}^T\, \|\Pi_{\mathcal{P}_k}(\p_k + \lambda \nabla q_k(\p_k))   - \p_k\|^2 \leq \frac{1}{\underline{a} T}\|\p_1 - \p^*_1\|^2 + \frac{\delta}{\underline{a}}\,(4 X + \delta);
\end{equation} 
\item[(b)] Let the functions $f_k$ be strongly smooth, for all $k$, with constant $M_k$, and let the smallest singular value of $\A$, $\sigma_{\min}$, be positive. Let $\mathcal{X} = \mathbb{R}^n$. In this case the $-q_k$ is strongly convex with constant $\sigma_{\min}^2/M_k$, and the operator $\Pi_{\mathcal{P}_k}(\opI + \lambda \nabla q_k)$ is a contraction, with $L_k = \max\{|1-\lambda \sigma_{\min}^2/M_k|,|1-\lambda \sigma^2_{\max}/m_k|\}$ and we obtain dual convergence in the sense of~\eqref{result2}: 
\begin{align}\label{result3.2}
\|\p_k - \p_k^*\| &\leq \hat{L}_{k}\,\|\p_1 - \p_1^*\| + \frac{1-\bar{L}_k^{k-1}}{1-\bar{L}_k} \delta, \\  L_k &=\max\{|1-\lambda \sigma_{\min}^2/M_k|,|1-\lambda \sigma^2_{\max}/m_k|\}. \nonumber 
\end{align}
In addition, we obtain also primal convergence as,
\begin{equation}\label{result3.3}
\|\x_k - \x_k^*\| \leq \frac{\sigma_{\max}}{m_k}\,\|\p_k - \p_k^*\|.
\end{equation}
\end{enumerate}
\end{corollary}

\begin{proof}
The proof is similar to the one of Corollary~\ref{co:gradient}. The only differences are that we are dealing with the dual functions $q_k$. For any dual function $q_k$, we have $\partial q_k(\p) = \A\x^*_{\p} - \b$, while for optimality $(\partial f_k + N_{\mathcal{X}})(\x^*_{\p}) + \A^\transp \p \ni \mathbf{0}$. Using the same notation as Lemma~\ref{lemma:scsm}, we have   $\partial f_{e,k}(\x^*_{\p}) + \A^\transp \p \ni \mathbf{0}$, or $\x^*_{\p} \in \partial \cf_{e,k} (-\A^\transp \p)$. Therefore, $\partial q_k(\p) = \A\partial \cf_{e,k} (-\A^\transp \p) - \b$, and thus if $f_k$ is strongly convex over $\mathbb{R}^n$ with constant $m_k$, then, by Lemma~\ref{lemma:scsm}, we have that $\cf_{e,k}$ is strongly smooth with constant $1/m_k$, and $-q_k$ is strongly smooth with constant $\sigma_{\max}^2/m_k$. In addition, if $f_k$ is strongly smooth over $\mathbb{R}^n$ with constant $M_k$ and $\mathcal{X} = \mathbb{R}^n$, then, by Lemma~\ref{lemma:scsm}, we have that $\cf_{e,k}$ is strongly convex with constant $1/M_k$, and $-q_k$ is strongly convex with constant $\sigma_{\min}^2/m_k$, see also~\cite{Ryu2015}. By applying this correspondence, the results~\eqref{result3.1}-\eqref{result3.2} follow. 

Result~\eqref{result3.3} is an application of~\cite[Theorem 2F.9]{Dontchev2009} applied to the generalized equation $(\partial f_k + N_{\mathcal{X}})(\x^*_{\p}) + \A^\transp\p \ni \mathbf{0}$.
\qed 
\end{proof}
\vskip2mm

\section{A bounding procedure for more general time-varying algorithms}\label{sec:opt.b}

In this section, we widen the class of optimization problems we tackle. In particular, we consider problems that do not give rise to \emph{bounded} operators when put in terms of fixed point equations. These optimization problems are, for example, the linearly constrained ones. To say it in another way, in this section we develop algorithms for time-varying operators that do not satisfy Assumption~\ref{as:bounded} directly, yet we force the boundedness requirement via a bounding procedure.

Recall the running \Mk iteration: 
\begin{equation}
\x_{k+1} = \opT_k(\x_k)
\end{equation}
for a properly defined $\alpha$-averaged operator sequence $\{\opT_k\}_{k \in \mathbb{N}_{>0}}$. We now assume that each $\opT_k$ is not necessarily bounded, that is $\im \opT \subseteq \mathbb{R}^n$. Define a proper, convex and compact set $\mathcal{B} \subset \mathbb{R}^n$. We introduce the bounded running \Mk iteration as the one that implements the recursion:
\begin{equation}\label{eq:brun}
\x_{k+1} = \Pi_\mathcal{B}\opT_k(\x_k).
\end{equation}
Since the projection operator $\Pi_\mathcal{B}$ is an $\alpha$-averaged operator, then due to Proposition~\ref{composition} on the composition of $\alpha$-averaged operators, the bounded running \Mk iteration~\eqref{eq:brun} converges under the same conditions of Theorem~\ref{theo:1} (by substituting $\opT_k$ with $\Pi_\mathcal{B}\opT_k$ and by noticing that $\Pi_\mathcal{B}\opT_k$ is bounded and verifies Assumption~\ref{as:bounded}). 

The seemingly ad-hoc bounding procedure introduced in~\eqref{eq:brun} has been used to bound Lagrangian multipliers in different contexts in the literature. In~\cite{Erseghe2015}, the author calls it a clipping procedure while in~\cite{Koppel2015} the authors assume the existence of a bounding set $\mathcal{B}$ in their Assumption~5; finally in~\cite{Bravo2016}, it is considered -- at least in Hilbert spaces -- to project an inexact update onto the domain (see iteration (IKM$_{p}$)).

With this in place, we can now tackle more general convex optimization problems, namely the ones with equality constraints.  

\subsection{Dual ascent with equality constraints}

Let us consider the problem,
\begin{equation}
\minimize_{\x \in \mathcal{X}} \, f(\x; t), \, \textrm{subject to} \, \A \x = \b
\end{equation}
where $f: \mathbb{R}^n\times \mathbb{R}_{+} \to \mathbb{R}$ is uniformly a CCP function, while $\A \in \mathbb{R}^{m\times n}$ and $\b \in \mathbb{R}^m$. The convex set $\mathcal{X} \subseteq \mathbb{R}^n$. We sample the optimization problem at $t_k$, $k = 1, 2, \dots$ and we assume that strong duality holds for each $t_k$. The dual problems of the sampled primal ones are
\begin{equation}
\maximize_{\p \in \mathbb{R}^m} \, q_k(\p),
\end{equation}
where $\p$ is the vector collecting the dual variables associated with the inequality constraint $\A\x = \b$. As part of the bounding procedure, we construct a convex compact set $\mathcal{B}$, such that the optimal dual variables are contained in it (one could start with a large enough $\mathcal{B}$ and then reduce it if possible).  

The running Lagrangian dual ascent scheme to find and track the time-varying dual optimal variables of~\eqref{eq.dual} is the following recursion,
\begin{enumerate}
\item[1.] Set $\p_1$ arbitrarily,
\item[2.] \textbf{for} $k>1$ \textbf{do}:

Compute the next approximate fixed point
\begin{equation}\label{tvdual2}
\p_{k+1} = \Pi_{\mathcal{B}}(\opI + \lambda \partial q_k)(\p_k) = \left\{\begin{array}{l}\x_{k+1} \in \arg\min_{\x \in \mathcal{X}} \{f_k(\x) + \p^\transp (\A\x - \b) \} \\ \p_k = \Pi_{\mathcal{B}}(\p_k + \lambda (\A\x_{k+1} - \b))\, .  \end{array}\right.
\end{equation}
\end{enumerate}

The running Lagrangian dual ascent~\eqref{tvdual2} converges as dictated in Corollary~\ref{co:da}, where now the user defined $\mathcal{B}$ is used in lieu of $\mathcal{P}_k$. Note that the set $\mathcal{B}$ is only needed in the case \emph{(a)} of Corollary~\ref{co:da}, while $\mathcal{B}$ can be chosen as $\mathbb{R}^m$ if we are in case \emph{(b)}, i.e., strongly convex functions and $\mathcal{X} = \mathbb{R}^n$. 

An additional result for the case of strongly smooth functions $f_k$ is reported next. This result is useful in practice when the matrix $\A$ is not full-row rank, such that $\sigma_{\min}(\A) = 0$. This case has been studied in \cite{Jakubiec2013} and in \cite{Necoara2015}.

\begin{corollary}\label{co:daeq}
Consider the running Lagrangian dual ascent defined in~\eqref{tvdual2} and the generated sequence $\{\x_k, \p_k\}_{k \in \mathbb{N}_{>0}}$ with $\mathcal{B} = \mathbb{R}^m$, $\b \in \im\A$ (so that the problem has a feasible solution), and $\mathcal{X} = \mathbb{R}^n$. Let the maximum singular value of $\A$ be $\sigma_{\max}$, and assume it is positive w.l.g.\,. Let $f_k$ be strongly convex for all $k$ with constant $m_k$ (which implies the dual function being differentiable, $\partial q_k = \nabla q_k$, and $-q_k$ being strongly smooth, with constant $\sigma^2_{\max}/m_k$ ).  Fix $\lambda \in (0, 2 m/\sigma^2_{\max})$ so that the operator $\opI + \lambda \nabla q_k$ is $\alpha$-averaged. Let Assumption~\ref{as:tv} hold.

Let the functions $f_k$ be strongly smooth, for all $k$, with constant $M_k$. In addition, assume that the initial dual variable is in the image of $\A$: $\p_1 \in \im \A$ and let $\sigma_0>0$ be the first (and minimal) nonzero singular value of $\A$.

In this case, 
the operator $\opI + \lambda \nabla q_k$ is a contraction for every $\p \in \im \A$, with $L_k = \max\{|1-\lambda \sigma_{0}^2/M_k|,|1-\lambda \sigma^2_{\max}/m_k|\}$ and we obtain dual convergence in the sense of~\eqref{result2}:
\begin{equation}\label{result3.2bis}
\|\p_k - \p_k^*\| \leq \hat{L}_{k}\,\|\p_1 - \p_1^*\| + \frac{1-\bar{L}^{k-1}_k}{1-\bar{L}_k} \delta, \, L_k =\max\{|1-\lambda \sigma_{0}^2/M_k|,|1-\lambda \sigma^2_{\max}/m_k|\}.
\end{equation}
In addition, we obtain also primal convergence as,
\begin{equation}\label{result3.3bis}
\|\x_k - \x_k^*\| \leq \frac{\sigma_{\max}}{m_k}\,\|\p_k - \p_k^*\|.
\end{equation}
\end{corollary}

\begin{proof}
To prove the contraction property, we only need to show that the functions $-q_k$ have a strong convex-like property for all $\p \in \im \A$ and that the Algorithm~\eqref{tvdual2} generates $\p_k \in \im \A$ (i.e., keeps the dual variable feasible). The second claim is easy to show since $\p_1 \in \im \A$ and 
\begin{equation}
\p_{k+1} = \p_{k} + \alpha (\A \x_{k+1}-\b) \in \im \A.  
\end{equation}
To show the first claim, we recall that $\partial q_k(\p) = \A\partial \cf_{e,k} (-\A^\transp \p) - \b$ (as proved in the proof of Corollary~\ref{co:da}), and $\partial \cf_{e,k}$ is strongly convex with constant $1/M_k$ (since $\mathcal{X} = \mathbb{R}^n$ and therefore $\partial \cf_{e,k} = \partial \cf_{k}$). Therefore, for all $\p,\q\in\im\A$:
\begin{multline}
(\partial q_k(\p) - \partial q_k(\q))^\transp (\q - \p) = (\partial \cf_{k} (-\A^\transp \p) - \partial \cf_{k} (-\A^\transp \q))^\transp \A^\transp (\q - \p) \geq \\ \sigma_0^2/M_k \|\p - \q\|^2,  
\end{multline}
which implies strong monotonicity of $-\partial q_k(\p)$ for all $\p \in \im \A$, and therefore strong convexity of $-q_k(\p)$ for all $p \in \im \A$.  Then the contraction property follows from the fact that $-q_k$ is both strongly smooth with constant $\sigma^2_{\max}/m_k$ as easy to show, and strongly convex (over the restricted domain). The rest follows as in the proof of Corollary~\ref{co:da}. 
\qed
\end{proof}

\vskip2mm
The result has been applied to time-varying distributed optimization, namely dual decomposition~\cite{Jakubiec2013}, although with a slightly different analysis technique. It is worth noting that iteration~\eqref{tvdual2} extends the work of \cite{Jakubiec2013} to constrained problems ($\mathcal{X} \subset \mathbb{R}^n$) and in the case of nonsmooth $f_k$.




\subsection{Douglas-Rachford splitting for composite optimization}

Consider once again the composite optimization problem
\begin{equation}\label{composite1}
\minimize_{\x \in \mathcal{X}(t)} \, f(\x; t) + g(\x;t), 
\end{equation}
where $f,g: \mathbb{R}^n\times \mathbb{R}_{+} \to \mathbb{R}$ are uniformly a CCP function, while the convex set $\mathcal{X}(t) \subseteq \mathbb{R}^n$ uniformly. Sample the optimization problem at $t_k$, $k = 1, 2, \dots$ and consider the following running version of the Douglas-Rachford splitting method, appropriately bounded:
\begin{enumerate}
\item[1.] Set $\z_1$ arbitrarily,
\item[2.] \textbf{for} $k>1$ \textbf{do}:

Compute the next approximate fixed point
\begin{align}\label{tvdr}
\z_{k+1} &= \Pi_{\mathcal{B}}(1/2 \opI + 1/2 \opC_{\partial f_{e,k}}\opC_{\partial g_{k}})\z_{k}\\
\x_{k+1} &= \opR_{\partial g_{k}}(\z_{k+1}).
\end{align}
\end{enumerate}
where $\opC_\opT$ is the Cayley operator of $\opT$, defined as $\opC_\opT :=2\opR_\opT - \opI$. 

In particular, the iteration~\eqref{tvdr} can be written as~\cite{Ryu2015} 
\begin{align}
\z_{k+1} &= \Pi_{\mathcal{B}}(\z_{k} + \prox_{\partial f_{e,k}, \lambda}(2 \x_{k} - \z_k) - \x_{k})\\
\x_{k+1} &= \prox_{\partial g_{k}, \lambda}(\z_{k+1}) \label{second}
\end{align}

\begin{corollary}\label{co:dr}
Consider the running Douglas-Rachford splitting defined in~\eqref{tvdr} and the generated sequence $\{\z_k, \x_k\}_{k \in \mathbb{N}_{>0}}$. For all CCP $f_k$ and $g_k$ the operator $(1/2 \opI + 1/2 \opC_{\partial f_{e,k}}\opC_{\partial g_{k}})$ is $\alpha$-averaged with $\alpha = 1/2$. Let Assumption~\ref{as:tv} hold.
\begin{enumerate}
\item[(a)] The sequence $\{\z_k,\x_k\}_{k \in \mathbb{N}_{>0}}$ converges in the sense of~\eqref{result1} and in particular,
\begin{equation}\label{result1.6}
\frac{1}{T}\,\sum_{k=1}^T\, \|\Pi_{\mathcal{B}}(1/2 \opI + 1/2 \opC_{\partial f_{e,k}}\opC_{\partial g_{k}})\z_{k}  - \z_k\|^2 \leq \frac{2}{T}{\|\z_1 - \z^*_1\|^2} + 2 \delta\,(4 X + \delta)
\end{equation} 
and
\begin{equation}\label{result2.7_0}
\|\x_k - \x_k^*\| \leq \|\z_k - \z_k^*\|.
\end{equation}
where $X$ is the uniform upper bound on $\|\mathcal{B}\|$.  
\item[(b)] Let the function  $f_{k}$ be strongly monotone and strongly smooth uniformly over $\mathbb{R}^n$, with constants $m_k$ and $M_k$, respectively, and let $\mathcal{X}_k = \mathbb{R}^n$. Then the operator $(1/2 \opI + 1/2 \opC_{\partial f_{e,k}}\opC_{\partial g_{k}})$ is a contraction with $L_k = 1/2(1 + \max\{\frac{\lambda M_k -1}{\lambda M_k +1}, \frac{1-\lambda m_k}{1+\lambda m_k}\})<1$ and we obtain primal convergence in the sense of~\eqref{result2}: 
\begin{align}\label{result2.6}
\|\z_k - \z_k^*\| &\leq \hat{L}_{k}\,\|\z_1 - \z_1^*\| + \frac{1-\bar{L}^{k-1}_{k}}{1-\bar{L}_{k}} \delta,\\ L_k &= 1/2\left(1 + \max\left\{\frac{\lambda M_k -1}{\lambda M_k +1}, \frac{1-\lambda m_k}{1+\lambda m_k}\right\}\right) \nonumber;
\end{align}
and 
\begin{equation}\label{result2.7}
\|\x_k - \x_k^*\| \leq \|\z_k - \z_k^*\|.
\end{equation}
Note that in this second case, one can pick $\mathcal{B} = \mathbb{R}^n$.
\end{enumerate}
\end{corollary}

\begin{proof}
Case \emph{(a)}. The Cayley operator is non-expansive, so the operator $(1/2 \opI + 1/2 \opC_{\partial f_{e,k}}\opC_{\partial g_{k}})$ is $\alpha$-averaged with $\alpha = 1/2$. The claim~\eqref{result1.6} follows considering the composition with the projection operator. In particular the operator $\Pi_{\mathcal{B}}(1/2 \opI + 1/2 \opC_{\partial f_{e,k}}\opC_{\partial g_{k}})$ is $\alpha$-averaged with $\alpha = 2/3$, while $(1-\alpha_k)/\alpha_k  = 1/2$.

Case \emph{(b)}. The contraction properties follows from~\cite[Theorems~1 and 2]{Giselsson2017}. Results~\eqref{result2.7_0}-\eqref{result2.7} follows from~\eqref{second} and the non-expansive nature of the resolvent.   
\qed
\end{proof}

\subsection{Alternating direction method of multipliers}\label{sec:admm}

We finish our analysis of time-varying algorithms with the celebrated alternating direction method of multipliers (ADMM) in its running form. We will rely on the fact that ADMM can be derived from the Douglas-Rachford splitting and use the results of the previous subsection.  

In this context, we are now interested in the time-varying problem,
\begin{equation}
\minimize_{\x \in \mathcal{X}(t), \z \in \mathcal{Z}(t) } f(\x; t) + g(\z; t)\,\, \textrm{subject to} \,\, \A \x + \B\z = \c,
\end{equation}
where $\A$, $\B$, $\c$ are matrices and vector of appropriate dimensions.  

By sampling the problem at instances $t_k$ with $k = 1,2,\dots$ we obtain a sequence of time-invariant problems, 
\begin{equation}\label{admm1}
\minimize_{\x \in \mathcal{X}_k, \z \in \mathcal{Z}_k } f_k(\x) + g_k(\z)\,\, \textrm{subject to} \,\, \A \x + \B\z = \c.
\end{equation}
Assume strong duality holds and write the dual problem of~\eqref{admm1} as
\begin{equation}\label{admmdual}
\maximize_{\bnu} -\cf_{e,k}(-\A^\transp \bnu) - \cg_{e,k}(-\B^\transp\bnu) + \c^\transp \bnu.
\end{equation}
Apply now the running bounded Douglas-Rachford splitting~\eqref{tvdr} to~\eqref{admmdual} with the splitting $-\cf_{e,k}(-\A^\transp \bnu) + \c^\transp \bnu$ and $- \cg_{e,k}(-\B^\transp\bnu)$ to obtain the following recursion (the detailed derivation is deferred in the Appendix). 

\begin{enumerate}
\item[1.] Set $\x_1, \z_1, \p_1$ arbitrarily,
\item[2.] \textbf{for} $k>1$ \textbf{do}:

Compute the next approximate fixed point
\begin{equation}\label{tvadmm}
\begin{array}{l}\z_{k+1} =\arg\min_{\mathcal{Z}_k}\{g_k(\z) + \p_k^\transp \B\z + \frac{\lambda}{2}\|\A\x_{k} + \B\z- \c\|^2 \}\\ 
\x_{k+1} =\arg\min_{\mathcal{X}_k}\{ f_k(\x) + \p_k^\transp \A\x + \frac{\lambda}{2}\|\A\x +\A\x_{k} + 2 \B\z_{k+1}- 2\c\|^2 \}\\
\p_{k+1} = \Pi_{\mathcal{B}}(\p_k + \lambda (\A\x_{k+1}+\A\x_{k} +\B z_{k+1}- 2\c)) - \lambda (\A\x_{k+1} - c)\,. \end{array}
\end{equation}
and $\bnu_{k+1} = \p_{k} + \lambda(\A\x_{k+1} + \B\z_{k+1} - \c)$.
\end{enumerate}

We note that~\eqref{tvadmm} is not the usual ADMM iteration, yet it correctly captures the need for a bounding procedure. In the case $\mathcal{B} = \mathbb{R}^m$ then one regain (as expected) a time-varying version of the standard ADMM, 

\begin{enumerate}
\item[1.] Set $\x_1, \z_1, \p_1$ arbitrarily,
\item[2.] \textbf{for} $k>1$ \textbf{do}:

Compute the next approximate fixed point
\begin{equation}\label{tvadmm1}
\begin{array}{l}\x_{k+1} =\arg\min_{\mathcal{X}_k}\{ f_k(\x) + \p_k^\transp (\A\x + B\z_k - \c) + \frac{\lambda}{2}\|\A\x + \B\z_{k}- \c\|^2 \} \\ 
\z_{k+1} =\arg\min_{\mathcal{Z}_k}\{ g_k(\z) + \p_k^\transp (\A\x_{k+1} + B\z - \c) + \frac{\lambda}{2}\|\A\x_{k+1} + \B\z- \c\|^2 \}\\
\bnu_{k+1} = \p_{k+1} = \p_k + \lambda (\A\x_{k+1} +\B z_{k+1}- \c)\, .  \end{array}
\end{equation}
\end{enumerate}

Convergence in the sense of Corollary~\ref{co:dr} for~\eqref{tvadmm}-\eqref{tvadmm1} follows in terms of a dual supporting sequence $\bbeta$ and the dual sequence $\bnu$ (doing the job of $\z$ and $\x$ in Corollary~\ref{co:dr}). The exact details are omitted in the interest of space, yet we report that in the case of strongly convex/strongly smooth $f_k$ one can obtain the following result. 

\begin{corollary}\label{co:admm}
Consider the running ADMM~\eqref{tvadmm1} and the generated sequence $\{\x_k, \z_k, \p_k\}_{k \in \mathbb{N}_{>0}}$. 
%
%
Let Assumption~\ref{as:tv} hold. 
Let the functions $f_k$ be strongly convex with constant $m_k$ and strongly smooth with constant $M_k$ and let $\mathcal{X}_k = \mathbb{R}^n$. Let $\sigma_{\max}$ and $\sigma_{\min}$ the maximum and minimum singular value of $\A$, supposed positive. In this case, the time-varying ADMM~\eqref{tvadmm1} is a contraction, and we obtain dual convergence as: 
\begin{align}\label{result7.2}
\|\p_k - \p_k^*\| &\leq \hat{L}_{k}\,\|\bbeta_1 - \bbeta_1^*\| + \frac{1-\bar{L}^{k-1}_k}{1-\bar{L}_k} \delta, \\ L_k &=1/2\left(1 + \max\left\{\frac{\lambda \sigma_{\max}^2/m_k -1}{\lambda \sigma_{\max}^2/m_k +1}, \frac{1-\lambda \sigma_{\min}^2/M_k}{1+\lambda \sigma_{\min}^2/M_k}\right\}\right), \nonumber
\end{align} 
where $\bbeta_1$ is the supporting variable for~\eqref{admmdual} (doing the job of $\z$ in Corollary~\ref{co:dr}).
\end{corollary}

\begin{proof}
The result follows from~\eqref{result2.6}, and \cite[Corollary~2]{Giselsson2017}. 
\qed
\end{proof} 

Corollaries~\ref{co:dr} and~\ref{co:admm} extend previous work on running ADMM~\cite{Ling2013}, which dealt with the strong convex-strong smooth case, for unconstrained consensus problems. Here we have a much more general characterization. 


\section{Objective convergence for primal time-varying optimization}\label{sec:obj_conv}

We focus now on objective convergence of the primal running algorithm we have presented. Different from fixed point residual convergence (and primal variable convergence under some circumstances), objective convergence cannot be directly derived from Theorem~\ref{theo:1} and \MK's arguments alone (although these results are needed). To tackle objective convergence, one needs an handle on how the function and its derivatives behave locally around a point $\x$, e.g., a Lipschitz descent lemma. 

We develop here results for the running projected gradient~\eqref{tvprojgradient}, proximal point~\eqref{tvfb}, and forward-backward splitting~\eqref{tvprox}, for which such a descent lemma is available. We leave for future research the other (dual) methods. 

We develop the theory in an unified framework, by leveraging the following (known) results. 

\begin{lemma}\emph{(Equivalence of primal methods)}\cite[Section~3.3]{Davis2014}\label{lemma.equival}
The running versions of the projected gradient algorithm~\eqref{tvprojgradient} and the proximal point algorithm~\eqref{tvprox} are special cases of the running forward-backward algorithm~\eqref{tvfb} with the following specifications:
\begin{itemize}
\item In the forward-backward algorithm~\eqref{tvfb}, put the function $g(\x;t) = 0$ to obtain the running projected gradient;  
\item In the forward-backward algorithm~\eqref{tvfb}, put the functions $f(\x;t) = 0$ and $g(\x;t) = f(\x;t)$ (where the second $f(\x;t)$ is the one of the proximal point algorithm) to obtain the proximal point algorithm. 
\end{itemize} 
\end{lemma}

\begin{lemma}\emph{(Joint decrease lemma)}\label{lemma.joint}
Define $h_k(\x): = f_k(\x) + g_k(\x)+i_{\mathcal{X}_k}$ and consider the running forward-backward algorithm~\eqref{tvfb} and the generated sequence $\{\x_{k}\}_{k \in \mathbb{N}_{>0}}$. Let the function $f_k$ be strongly smooth for all $k$ with constant $M_k \in [0,M]$. Then one has, 
\begin{multline}\label{lemma.eq0}
h_k(\x_{k+1})  - h_k(\x^*_k) \leq \frac{1}{2 \lambda}\left(\|\x_{k} - \x_{k}^* \|^2 - \|\x_{k+1} - \x_{k}^* \|^2 \right) + \\\left(\frac{M_k}{2}-\frac{1}{2\lambda} \right) \|\x_{k+1} - \x_{k}\|^2.
\end{multline}
\end{lemma}

\begin{proof}
Follows directly from the proof of~\cite[Theorem~3]{Davis2014}, which uses a joint Lipschitz descent lemma. 
\qed
\end{proof}

To handle a time-varying cost function, we will need to fix a \emph{universal} scaling: the cost function will change continuously in time as well as the optimizer set, however, the optimal value can (and will) be considered constant (i.e., one can rescale or shift the cost functions, such that the optimal value is constant in time without loss of generality). 

We will further assume the following. 

\begin{assumption}\emph{(Bounded functional changes)}\label{as.new}
The cost function $h_{k}(\x)$'s changes in time are upper bounded by a finite scalar $\sigma \geq 0$, as
$$
|h_{k+1}(\x) - h_{k}(\x)| \leq \sigma, \textrm{ for all } \x \in \textrm{dom}(h_k) \cap \textrm{dom}(h_{k+1}), \textrm{ and for all } k.
$$
\end{assumption}

Assumption~\ref{as.new} is an additional assumption w.r.t. Assumption~\ref{as:tv}, that pertains the variations of the cost function. It is rather easy to see that Assumption~\ref{as:tv} does not implies Assumption~\ref{as.new}, and therefore the latter is needed. Assumption~\ref{as.new} allows one to track how the cost function changes in time, by measuring the variations at specific points $\x$ in the domain of the function. Since $h_k(\x): \mathbb{R}^n \to \mathbb{R}\cup{\{+\infty\}}$, to make Assumption~\ref{as.new} hold, $\x$ has to belong to the domain of $h_{k}$ and $h_{k+1}$, which means that at least $\x \in \mathcal{X}_{k+1} \cap \mathcal{X}_{k}$, for all $k$. A special case, which we will consider here is that $\mathcal{X}_k$ is time-invariant, and thus the domain of $i_{\mathcal{X}_k}$ is the same for all $k$. 

We are now ready for the main result of this section. 

\begin{proposition}\label{co:fb_obj}
Consider the running forward-backward splitting as it has been defined in~\eqref{tvfb} and the generated sequence $\{\x_k\}_{k \in \mathbb{N}_{>0}}$. Define $F_k(\x):= f_k(\x) + g_k(\x)$. Let $f_k$ be strongly smooth for all $k$ with constant $M_k\in [0,M]$ (which implies $\partial f_k = \nabla f_k$).  Fix $\lambda \in (0, 2/M)$. Let Assumption~\ref{as:tv} hold.
Let the sets $\mathcal{X}_k$ be compact, and therefore bounded, and let $X$ be defined as $X = \max_{k} \|\mathcal{X}_k\|$.  Further assume that the sets $\mathcal{X}_k$ are time-invariant, i.e.,  $\mathcal{X}_k = \mathcal{X}$, and let Assumption~\ref{as.new} hold for a certain $\sigma\geq0$. 
Define $\underline{a}:= \min_k (1-\alpha_k)/\alpha_k$, where $\alpha_k = 1/(2-\lambda M_k/2)$. Then, the objective sequence $\{F_k(\x_k)\}_{k \in \mathbb{N}_{>0}}$ converges as
\begin{multline}\label{obj_result1}
\frac{1}{T}\,\sum_{k=1}^T\, F_{k+1}(\x_{k+1})  - F_{k+1}(\x^*_{k+1}) \leq \\ \left\{
\begin{array}{lr}
\displaystyle\frac{1}{2\lambda T} \|\x_{1} - \x_{1}^* \|^2 + \frac{\delta}{2 \lambda} (4 X + \delta) + \sigma, & \qquad \textrm{for } \lambda \leq 1/M_k,\\ &\\
\displaystyle\frac{C}{2\lambda T}\|\x_{1} - \x_{1}^* \|^2 + \frac{\delta C }{{2 \lambda}} (4 X + \delta) + \sigma, & \qquad  \textrm{otherwise,}
\end{array}
\right.\, 
\end{multline} 
with $C = \left(1 + \frac{\lambda M_k - 1}{\underline{a}}\right)$.
\end{proposition}

\begin{proof}
We start from~\eqref{lemma.eq0}, multiply by $2 \lambda >0$ and taking the average over time $T$, 
\begin{multline}\label{lemma.eq1}
\frac{1}{T}\,\sum_{k=1}^T\, 2 \lambda( h_k(\x_{k+1})  - h_k(\x^*_k)) \leq \frac{1}{T}\,\sum_{k=1}^T\, \left(\|\x_{k} - \x_{k}^* \|^2 - \|\x_{k+1} - \x_{k}^* \|^2 \right) + \\ \frac{1}{T}\,\sum_{k=1}^T\,  (\lambda{M_k}-1) \|\x_{k+1} - \x_{k}\|^2.
\end{multline}
By using the same development of the proof of Theorem~\ref{theo:1} and in particular Equations~\eqref{dummy:0} till \eqref{dummy:6969}, we can bound 
\begin{equation}\label{bound-1}
\|\x_{k} - \x_{k}^* \|^2 - \|\x_{k+1} - \x_{k}^* \|^2 \leq \|\x_{k} - \x_{k}^* \|^2 - \|\x_{k+1} - \x_{k+1}^* \|^2 + \delta (4 X + \delta).
\end{equation}
In addition, by Corollary~\ref{co:fb} part (a), 
\begin{equation}\label{bound-2}
\frac{1}{T}\,\sum_{k=1}^T\, \|\x_{k+1} - \x_{k}\|^2 \leq \frac{1}{\underline{a}T}{\|\x_1 - \x^*_1\|^2} + \frac{\delta}{\underline{a}}\,(4 X + \delta),
\end{equation}
where we have substituted $\x_{k+1} = \prox_{g_k,\mathcal{X}_k, \lambda}(\x_k - \lambda \nabla f_k(\x_k))$. Furthermore, by Assumption~\ref{as.new}, 
\begin{equation}\label{bound-3}
h_k(\x_{k+1})  - h_k(\x^*_k) \geq h_{k+1}(\x_{k+1})  - h_{k+1}(\x^*_{k+1}) - \sigma.
\end{equation}
By putting together the bounds~\eqref{bound-1}, \eqref{bound-2}, and~\eqref{bound-3} in~\eqref{lemma.eq1}, we obtain
\begin{multline}\label{lemma.eq1-0}
\frac{1}{T}\,\sum_{k=1}^T\, h_{k+1}(\x_{k+1})  - h_{k+1}(\x^*_{k+1}) \leq \frac{1}{2 \lambda T}\|\x_{1} - \x_{1}^* \|^2 + \frac{\delta}{2 \lambda} (4 X + \delta) + \\ \left(\frac{\lambda M_k - 1}{2 \lambda}\right)\left(\frac{1}{\underline{a}T}{\|\x_1 - \x^*_1\|^2} + \frac{\delta}{\underline{a}}\,(4 X + \delta)\right) + \sigma.
\end{multline}
By noticing that $\x_k$ is feasible for problem $k$ for all $k$, then $h_k \equiv F_k$, which yields the result. \qed
\end{proof}

From which the following corollaries can be readily obtained.

\begin{corollary}\label{co:gr-obj}
Under the same conditions of Proposition~\ref{co:fb_obj}, the running proximal point defined in~\eqref{tvprojgradient}, whose objective sequence is $\{f_k(\x_k)\}_{k \in \mathbb{N}_{>0}}$, converges as~\eqref{obj_result1} with $F_k = f_k$. 
\end{corollary}

\begin{proof}
Directly from~Proposition~\ref{co:fb_obj} and Lemma~\ref{lemma.equival}. \qed
\end{proof}

\begin{corollary}\label{co:pp-obj}
Under the same conditions of Proposition~\ref{co:fb_obj}, the running proximal point defined in~\eqref{tvprox}, whose objective sequence is $\{f_k(\x_k)\}_{k \in \mathbb{N}_{>0}}$, converges as
\begin{multline}\label{obj_result2}
\frac{1}{T}\,\sum_{k=1}^T\, f_{k+1}(\x_{k+1})  - f_{k+1}(\x^*_{k+1}) \leq 
\frac{1}{2\lambda T} \|\x_{1} - \x_{1}^* \|^2 + \frac{\delta}{2 \lambda} (4 X + \delta) + \sigma.
\end{multline} 
\end{corollary}

\begin{proof}
Directly from~Proposition~\ref{co:fb_obj} and Lemma~\ref{lemma.equival}, given that for the proximal point method $M_k = 0$ for all $k$, and therefore $\lambda M_k -1 \leq 0$ for all $k$. \qed
\end{proof}

Proposition~\ref{co:fb_obj} and Corollaries~\ref{co:gr-obj}-\ref{co:pp-obj} express the objective convergence of three of the presented running methods. In particular, convergence in average sense goes as $O(1/T)$ up to an error bound which depends on the variability of the optimization problem.


\section{Numerical example}\label{sec:num}

In this section, we display a numerical scenario depicting the behavior of the running version of ADMM that we have proposed in this paper, i.e.~\eqref{tvadmm}. The example is taken from a signal processing application: distributed time-varying localization via range measurement in wireless sensor networks. The example and its distributed implementation via convex relaxations and ADMM are developed in the time-invariant setting in~\cite{Simonetto2014}. Here, we only briefly present the problem and introduce its time-varying counterpart.

\subsection{Localization via range measurement}

We consider a network of $N$ static wireless sensor nodes with computation and communication capabilities, living in a $2$-dimensional space. We denote the set of all nodes $\mathcal{V} = \{1, \dots, N\}$. Let $\x_{(i)} \in \mathbb{R}^D$ be the position vector of the $i$-th sensor node, or equivalently, let $\X = [\x_{(1)} , \dots , \x_{(N)}] \in \mathbb{R}^{D\times N}$ be the matrix collecting the position vectors. We consider an environment with line-of-sight conditions between the nodes and we assume that some pairs of sensor nodes $(i,j)$ have access to noisy range measurements as	
\begin{equation}
r_{(i),(j)} = d_{(i),(j)} + \nu_{(i),(j)},
\label{eq.noisymeasurement}
\end{equation}
where $d_{(i),(j)} = \|\x_{(i)} - \x_{(j)}\|$ is the noise-free Euclidean distance and $\nu_{(i),(j)}$ is an additive noise term with known probability density function (PDF). We call the inter-sensor sensing PDF as $p_{(i),(j)}(d_{(i),(j)}(\x_{(i)},\x_{(j)})|r_{(i),(j)})$, where we have indicated explicitly the dependence of $d_{(i),(j)}$ on the sensor node positions $(\x_{(i)},\x_{(j)})$. 

In addition, we consider that some sensors also have access to noisy range measurements with some fixed anchor nodes (whose position $\a_{(l)}$, for $l \in \{1,\dots,M\}$, is known by all the neighboring sensor nodes of each $\a_{(l)}$) as  
\begin{equation}
v_{(i),(l)} = e_{(i),(l)} + \mu_{(i),(l)},
\label{eq.noisyanchormeasurement}
\end{equation}
where, $e_{(i),(l)} = \|\x_{(i)} - \a_{(l)}\|$ is the noise-free Euclidean distance and $\mu_{(i),(l)}$ is an additive noise term with known probability distribution. We denote as $p_{{(i),(l)},\mathrm{a}}(e_{(i),(l)}(\x_{(i)},\a_{(l)})|v_{(i),(l)})$ the anchor-sensor sensing PDF.

We use graph theory terminology to characterize the set of sensor nodes $\mathcal{V}$ and the measurements $r_{(i),(j)}$ and $v_{(i),(l)}$. In particular, we say that the measurements $r_{(i),(j)}$ induce a graph with $\mathcal{V}$ as vertex set, i.e., for each sensor node pair $(i,j)$ for which there exists a measurement $r_{(i),(j)}$, there exists an edge connecting $i$ and $j$. The set of all edges is $\mathcal{E}$ and its cardinality is $E$. We denote this undirected graph as $\mathcal{G} = (\mathcal{V}, \mathcal{E})$. The neighbors of sensor node $i$ are the sensor nodes that are connected to $i$ with an edge. The set of these neighboring nodes is indicated with $\mathcal{N}_i$, that is $\mathcal{N}_i = \{j| (i, j) \in \mathcal{E}\}$. Since the sensor nodes are assumed to have communication capabilities, we implicitly assume that each sensor node $i$ can communicate with all the sensors in $\mathcal{N}_i$, and with these only. In a similar fashion, we collect the anchors in the vertex set $\mathcal{V}_\mathrm{a} = \{1, \dots, M\}$ and we say that the measurements $v_{(i),(l)}$ induce an edge set $\mathcal{E}_\mathrm{a}$, composed by the pairs $(i,k)$ for which there exists a measurement $v_{(i),(l)}$. Also, we denote with $\mathcal{N}_{i, \mathrm{a}}$ the neighboring anchors for sensor node $i$, i.e., $\mathcal{N}_{i, \mathrm{a}} = \{l| (i, l) \in \mathcal{E}_\mathrm{a}\}$.

\textbf{Problem Statement.} The sensor network localization problem is formulated as estimating the position matrix $\X$ (in some cases, up to an orthogonal transformation) given the measurements $r_{(i),(j)}$ and $v_{(i),(l)}$ for all $(i,j)\in \mathcal{E}$ and $(i,l) \in \mathcal{E}_{\mathrm{a}}$, and the anchor positions $\a_{(l)}$, $l \in \mathcal{V}_{\mathrm{a}}$. The sensor network localization problem can be written in terms of maximizing the likelihood leading to the following optimization problem
\begin{multline}
\hskip-0.3cm\X_{\textrm{ML}}^* = \arg\hskip-0.1cm\max_{\hskip-0.4cm \X\in\mathbb{R}^{D\times N}} \left\{\sum_{(i,j) \in\mathcal{E}} \ln p_{(i),(j)}(d_{(i),(j)}(\x_{(i)},\x_{(j)})|r_{(i),(j)}) + \right.\\ \left.\sum_{(i,l)\in\mathcal{E}_{\mathrm{a}}} \ln p_{(i),(l),\mathrm{a}}(e_{(i),(l)}(\x_{(i)},\a_{(l)})|v_{(i),(l)})\right\}.
\label{eq.mle}
\end{multline}

The problem at hand is nonconvex and NP-Hard, even in the case of Gaussian noise. In~\cite{Simonetto2014}, we have proposed a technique to relax the problem into a convex semidefinite program and we have use ADMM to distribute the solution of this relaxed problem among the nodes themselves. In particular, each node, while communicating only with its neighbors can determine its own location. 

\subsection{Time-varying problem}

Here, we consider a (per-snapshot) time-varying extension of the problem, where we would like to solve the nonconvex
\begin{multline}
\hskip-0.3cm\X_{\textrm{ML}}^*(t) = \arg\hskip-0.1cm\max_{\hskip-0.4cm \X\in\mathbb{R}^{D\times N}} \left\{\sum_{(i,j) \in\mathcal{E}} \ln p_{(i),(j)}(d_{(i),(j)}(\x_{(i)},\x_{(j)})|r_{(i),(j)}(t)) +\right.\\\left.  \sum_{(i,l)\in\mathcal{E}_{\mathrm{a}}} \ln p_{(i),(l),\mathrm{a}}(e_{(i),(l)}(\x_{(i)},\a_{(l)}(t))|v_{(i),(l)}(t))\right\}.
\label{eq.mle}
\end{multline}
where now the measurements $r_{(i),(j)}(t)$ and $v_{(i),(l)}(t)$ as well as the anchor positions $\a_{(l)}(t)$ change in time.

We sample the problems at $t_k, k = 1, 2, \dots$ and for each of them, we proceed in the same way as~\cite{Simonetto2014} and produce ADMM iterations, which can be implemented in a distributed way. The resulting scheme is a per-snapshot running ADMM as~\eqref{tvadmm},  whose convergence is encoded in Section~\ref{sec:admm}. Note that the resulting convex problem in~\cite{Simonetto2014} is a constrained one, so one should apply~\eqref{tvadmm}. 

The numerical results of such setting are represented in Figures~\ref{fig.1} and \ref{fig.2} for the following settings: $8$ nodes and $5$ anchors randomly deployed in the box $[-.5, 0.5]^2$ (the maximum number of neighbors is $3$), Gaussian noise for all the measurements with the same standard deviation $\sigma = 0.1$. All the nodes and anchors are moving along a circular path center in the origin with angular speed $\omega$ (different in different simulations), and the sampling period is $1$. All the decision variables of the running ADMM are initialized at $0$, and $\lambda$ is picked as $0.3$. The set $\mathcal{B}$ is taken as the whole $\mathbb{R}^m$ (to show that in this computational example, the choice of $\mathcal{B}$ does not influence convergence). Further details on the simulation setup are given in~\cite{Simonetto2014}. 

As we can see from Figures~\ref{fig.1} and \ref{fig.2}, the proposed running ADMM converges in primal sense and eventually reaches an error floor. The tracking error is defined as 
\begin{equation}
\textrm{Tracking Error}(k) = \sum_{i=1}^{N} \|\x_{(i),k} - \x_{(i)}^*(k)\|^2,
\end{equation}
where $\x_{(i)}^*(k)$ is the centralized optimal solution at time $t_k$ and $\x_{(i),k}$ is the $k$-th iterate of the recursion~\eqref{tvadmm}. 

We see also how the nodes are able to find and track the optimizer of a time-varying optimization problem up to a bounded error (depending on the angular speed-- that is depending on the variability of the optimizers $\delta$), in a distributed fashion (i.e., by talking only to their neighbors). 

\begin{figure}
\psfrag{X}{\scriptsize $x$}
\psfrag{Y}{\scriptsize $y$}
\psfrag{t1}[c]{\scriptsize $k = 2$}
\psfrag{t2}[c]{\scriptsize $k = 4$}
\psfrag{t3}[c]{\scriptsize $k = 6$}
\psfrag{t4}[c]{\scriptsize $k = 8$}
\psfrag{t5}[c]{\scriptsize $k = 32$}
\psfrag{t6}[c]{\scriptsize $k = 64$}
\includegraphics[width=\textwidth, trim=0cm 0cm 0cm 0cm,clip=on]{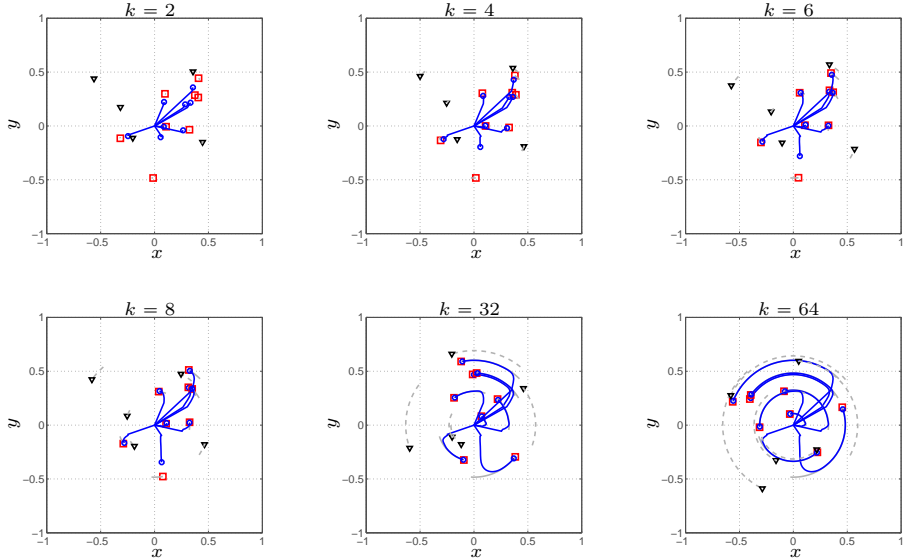}
\caption{Tracking capabilities of running ADMM with $\omega = \pi/100$. The different snapshots refer to different sampling instances $t_k$, $k = 2, 4, 6, 8, 32, 64$. The black triangles are anchor positions, the red squares are the node positions computed by an exact centralized optimization per snap-shot, the blue circles are the positions computed by the running ADMM. Lines represent the trajectories while the algorithm runs. }
\label{fig.1}
\end{figure}

\begin{figure}
\centering
\psfrag{x}[c]{Iterations $k$}
\psfrag{y}[c]{Tracking Error}
\psfrag{aaaa}[l]{$\omega = 0$ (static)}
\psfrag{bbbb}[l]{$\omega = \pi/200$ }
\psfrag{cccc}[l]{$\omega = \pi/100$ }
\psfrag{dddddddddd}[l]{$\omega = \pi/50$}
\includegraphics[width=.6\textwidth, trim=0cm 0cm 0cm 0cm,clip=on]{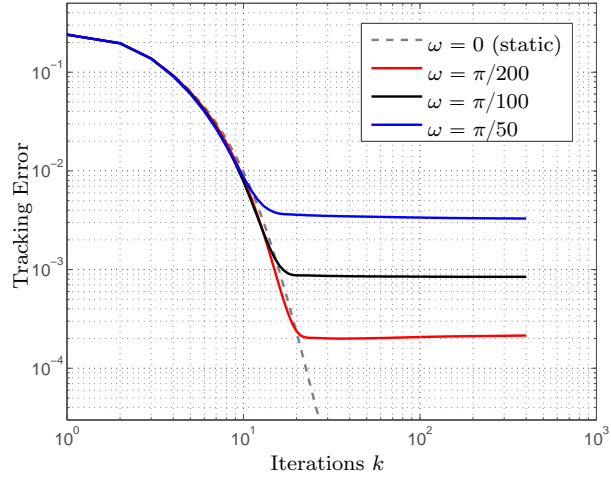}
\caption{Tracking error vs. number of iterations for different angular speeds $\omega$.}
\label{fig.2}
\end{figure}

\begin{remark}
The aim of the simulation results is to show that the theory developed in this paper can be applied to a fairly complex convex optimization problem, with semidefinite constraints, and running in a distributed fashion. More about the application example (especially in a mobile setting) can be found in~\cite{Jamali-Rad2012,Simonetto2014a}.
\end{remark}

\section{Conclusions}\label{sec:concl}

We have presented a general framework for time-varying optimization problems leveraging averaged operator theory. Our main meta-algorithm is the time-varying version of the fixed point algorithm, here renamed running \Mk algorithm. With this in place, we have derived a number of convergence results for running version of commonly used algorithms in convex optimization. 


\section*{Appendix: Proofs of Theorems~\ref{theo:1} and ~\ref{theo:3}}

Before tackling the proof of Theorem~\ref{theo:1} a technical lemma is needed.

\begin{lemma}\emph{(Triangle equality, \cite{Ryu2015})} For any scalar $\theta \in \mathbb{R}$, and vectors $\a,\b\in \mathbb{R}^n$, the following equality holds true
\begin{equation}\label{lemma:eq}
\|(1-\theta)\a + \theta \b\|^2 = (1-\theta)\|\a\|^2 + \theta \|\b\|^2 - \theta(1-\theta) \|\a - \b\|^2.
\end{equation}
\end{lemma}

\vskip2mm

\begin{proof} \emph{(Of Theorem~\ref{theo:1})}

\noindent Case \emph{(a)}. Starting with the basic iteration~\eqref{tvfixedpoint},
\begin{equation}
\|\x_{k+1} - \x\|^2 = \|\opT(\x_k) - \x\|^2 =  \|\x_k + \alpha_k (\opG_k(\x_k) - \x_k)  - \x\|^2,
\end{equation}
which is true for any $k \in \mathbb{N}_{>0}$ and any $\x$. Let $\x$ be in $\fix \opT_k = \fix \opG_k$, i.e., $\x = \x^*_k$. Then, since by definition of fixed point $\x_k^* = (1-\alpha)\x_k^* + \alpha \opG_k(\x^*_k)$, the right-hand side can be written as 
\begin{equation}
\|\x_{k+1} - \x^*_k\|^2 = \|(1 - \alpha_k)(\x_k - \x^*_k) + \alpha_k (\opG_k(\x_k) - \opG_k(\x_k^*))\|^2.
\end{equation}
By applying~\eqref{lemma:eq} and recalling that $\x_k^* = \opG_k(\x^*_k)$, we obtain the bound
\begin{multline}
\|\x_{k+1} - \x^*_k\|^2 \leq (1 - \alpha_k)\|\x_k - \x^*_k\|^2 + \\ \alpha_k \|\opG_k(\x_k) - \opG_k(\x_k^*)\|^2 - \alpha_k(1-\alpha_k)\|\opG_k(\x_k) - \x_k\|^2.
\end{multline}
Since $\opG_k$ is a nonexpansive operator, then $\|\opG_k(\x_k) - \opG_k(\x_k^*)\|^2 \leq \|\x_k - \x^*_k\|^2$, which yields,
\begin{equation}
\|\x_{k+1} - \x^*_k\|^2 \leq \|\x_k - \x^*_k\|^2 - \alpha_k(1-\alpha_k)\|\opG_k(\x_k) - \x_k\|^2,
\end{equation}
and rearranging
\begin{equation}\label{dummy:0}
\alpha_k(1-\alpha_k)\|\opG_k(\x_k) - \x_k\|^2 \leq \|\x_k - \x^*_k\|^2 - \|\x_{k+1} - \x^*_k\|^2.
\end{equation}
We focus now on the term $\|\x_{k+1} - \x^*_{k+1}\|^2$. By adding and subtracting any fixed point of $\opT_{k}$, we have
\begin{equation}
\|\x_{k+1} - \x^*_{k} - (\x^*_{k+1} - \x^*_{k})\|^2,
\end{equation}
which can be expanded as
\begin{multline}\label{dummy:69}
\|\x_{k+1} - \x^*_{k} - (\x^*_{k+1} - \x^*_{k})\|^2 = \|\x_{k+1} - \x^*_{k}\|^2 + \|\x^*_{k+1} - \x^*_{k}\|^2\\ - 2 (\x_{k+1} - \x^*_{k})^\transp(\x^*_{k+1} - \x^*_{k}),
\end{multline}
and upper bounded via Assumptions~\ref{as:tv}-\ref{as:bounded} as,
\begin{equation}\label{dummy70}
\|\x_{k+1} -\x^*_{k+1}\|^2 \leq  \|\x_{k+1} - \x^*_{k}\|^2 + \delta^2 + 4 X \delta,
\end{equation}
or equivalently
\begin{equation}\label{dummy:1}
- \|\x_{k+1} - \x^*_{k}\|^2 \leq - \|\x_{k+1} -\x^*_{k+1}\|^2 + \delta^2 + 4 X \delta,
\end{equation}
By substituting the upper bound~\eqref{dummy:1} into~\eqref{dummy:0}, we obtain 
\begin{equation}\label{dummy:6969}
\alpha_k(1-\alpha_k)\|\opG_k(\x_k) - \x_k\|^2 \leq \|\x_k - \x^*_k\|^2 - \|\x_{k+1} - \x^*_{k+1}\|^2 + 4 X \delta + \delta^2.
\end{equation}
If we now sum this inequality for all $k>0$ and discard the negative terms in the right-hand side, we obtain,
\begin{equation}\label{dummy:2}
\sum_{k=1}^T\alpha_k(1-\alpha_k)\|\opG_k(\x_k) - \x_k\|^2 \leq \|\x_1 - \x^*_1\|^2 + T(4 X \delta + \delta^2),
\end{equation}
and by dividing both sides by $T$ the claim~\eqref{result1} follows.

\vskip2mm

\noindent Case \emph{(b)}
The proof in this case is straightforward. Start by, 
\begin{align}\label{dummy:4}
\|\x_{k+1} - \x_{k+1}^*\| &= \|\x_{k+1} - \x_{k}^* - (\x_{k}^* - \x_{k+1}^*)\| = \nonumber \\ 
&= \|\opT_k(\x_k) - \opT_k(\x_k^*) - (\x_{k}^* - \x_{k+1}^*) \| \leq L_k \|\x_k - \x_k^*\| + \delta,
\end{align}
where we have use the contractive property of $\opT_k$ in Assumption~\ref{as:contraction}, and the triangle inequality. 

By iterating~\eqref{dummy:4} backward in time, we obtain
\begin{equation}
\|\x_{k} - \x_{k}^*\| \leq L_1 \cdots L_{k-1} \|\x_1 - \x^*_1\| + \frac{1 - \bar{L}^{k-1}_k}{1 - \bar{L}_k} \delta, 
\end{equation}
where $\bar{L}_k = \max_k{L_k}$. Thus, by defining $\hat{L}_k = L_1 \cdots L_{k-1}$, the claim~\eqref{result2} is proven.
\qed
\end{proof}

\vskip2mm


\begin{proof} \emph{(Of Theorem~\ref{theo:3})}

Direct by replacing~\eqref{dummy70} by the bound~\eqref{horder} and following the same steps till~\eqref{dummy:2}.  
\qed
\end{proof}

\section*{Appendix: Derivations of~\eqref{tvadmm}-\eqref{tvadmm1}}

The derivation of the recursion~\eqref{tvadmm} and \eqref{tvadmm1} from the Douglas-Rachford splitting applied to the dual~\eqref{admmdual} follows from~\cite[Page~35]{Ryu2015} with minor modifications. 

First of all, the update on their $y^k$, now reads $y^{k+1} = \Pi_{\mathcal{B}}[y^k+\xi^{k+1} - \zeta^{k+1}]$. With the substitutions: $\alpha \to \lambda$, $\tilde{z}^{k} \to \z_{k}$, $\tilde{x}^{k} \to \x_{k}$, and $y^k \to \y_k$, $\alpha u^k \to \p_k$, and finally $\y_k = \p_k + \lambda (\A \x_{k+1} - \c)$, then~\eqref{tvadmm} follows directly. Note that one cannot swap the order of $\p_{k+1}$ and $\x_{k+1}$ to obtain the standard ADMM, since now $\p_{k+1}$ depends on both $\x_{k}$ and $\x_{k+1}$. To obtain the dual variable $\bnu_{k+1}$, as for the Douglas-Rachford splitting \eqref{second} we have that is equivalent to their $\zeta^{k+1}$, that is 
\begin{equation}\label{dummy.nup}
\bnu_{k+1} = \y^{k} + \lambda \B \z_{k+1} = \p_k + \lambda (\A \x_{k+1} + \B \z_{k+1} - \c),
\end{equation}
from which the relation after~\eqref{tvadmm} follows. 

When $\mathcal{B} = \mathbb{R}^m$, then the steps of~\cite[Page~35]{Ryu2015} can be carried out till the end and the standard ADMM follows. In particular, due to~\eqref{dummy.nup}, $\bnu_{k+1} = \p_{k+1}$.

\section*{Acknowledgements}

The author wishes to thank Prof. Panagiotis (Panos) Patrinos at KULeuven and Dr. Adrien Taylor at UCLouvain for insightful discussions and suggestions on an early draft on the manuscript.

\footnotesize
\bibliographystyle{spmpsci}
\bibliography{../../../PaperCollection00}

\end{document}